\documentclass[a4paper,11pt,reqno]{amsart}

\usepackage{graphicx}
\usepackage{amsthm}
\usepackage{float}
\usepackage{amsmath,latexsym}
\usepackage{amssymb}
\usepackage{geometry}
\usepackage{amsfonts}
\usepackage{hyperref}
\usepackage[skip=4pt]{subcaption}
\usepackage[ruled]{algorithm2e} %[ruled,vlined, longend]
\usepackage{adjustbox}

\usepackage[authoryear]{natbib}
 \bibpunct[, ]{(}{)}{,}{a}{}{,}%
\usepackage{float}
\usepackage{hyperref}
\usepackage[skip=4pt]{subcaption}
\usepackage{multirow}
\usepackage{lscape}
\usepackage{bigstrut}
\usepackage{graphicx}
\usepackage{cancel}

\newcommand{\dsum}{\displaystyle\sum}

\usepackage{pgfplots}

\newtheorem{prop}{Proposition}[section]

\newtheorem{rmk}{Remark}[section]

\allowdisplaybreaks

\let\origmaketitle\maketitle
\def\maketitle{
  \begingroup
  \def\uppercasenonmath##1{} % this disables uppercasing title
  \let\MakeUppercase\relax % this disables uppercasing authors
  \origmaketitle
  \endgroup
}

\begin{document}

\title{\Large Hub Location with Protection under Link Failures}

\author[V. Blanco, E. Fern\'andez \MakeLowercase{and} Y. Hinojosa]{{\large V\'ictor Blanco$^\dagger$, Elena Fern\'andez$^\ddagger$ and  Yolanda Hinojosa$^\star$}\medskip\\
$^\dagger$Institute of Mathematics (IMAG), Universidad de Granada\\
$^\ddagger$ Dept. of Statistics \& OR, Universidad de C\'adiz\\
$^\star$Institute of Mathematics (IMUS), Universidad de Sevilla}

\address{IMAG, Universidad de Granada, SPAIN.}
\email{vblanco@ugr.es}

\address{Dpt. Stats \& OR, Universidad de C\'adiz, SPAIN.}
\email{elena.fernandez@uca.es}

\address{IMUS, Universidad de Sevilla, SPAIN.}
\email{yhinojos@us.es}

\date{}

\maketitle

\begin{abstract}
This paper introduces the Hub Location Problem under Link Failures, a hub location problem in which activated inter-hub links may fail with a given probability.
 Two different optimization models are studied, which construct hub backbone networks protected under hub links disruptions by imposing that for each commodity an additional routing path exists besides its original routing path.
Both models consider the minimization of the set-up costs of the activated hubs and inter-hub links plus the expected value of the routing costs of the original and alternative paths. The first model builds explicitly the alternative routing paths, while the second model guarantees that for each commodity at least one alternative path exists using a large set of connectivity constraints, although the alternative paths are not built explicitly. The results of extensive computational testing allow to analyze the performance of the two proposed models and to evaluate the extra cost required to design a robust backbone network under hub links failures. The obtained results support the validity of the proposal.
\keywords{Hub Location; Integer Programming; Robust Network Design; Disruptions}
\end{abstract}

\section{Introduction\label{sect1}}

Hub location lies in the intersection of Location Analysis and Network Design, and produces challenging optimization problems with multiple applications, mostly in the fields of distribution/logistics (parcel delivery, air transportation, etc.) and telecommunications {\citep{farahani2013hub}}. The increasing attention they have received in the last decades is thus not surprising \citep[see e.g.][]{Campbell25,ContrerasOkelly}.
One of the current trends in hub location is the search of models suitable for dealing with different sources of uncertainty \citep{Alumur}. While some models in the literature consider uncertainty in demand \citep{Contreras1,Contreras2}, other models are concerned with the robustness of solution {hub} networks, by associating uncertainty with the possibility (probability) of disruption of the {involved} elements of {the} solution networks and looking for solutions that are \emph{robust} under disruptions.

Some works have studied models in which it is assumed that activated hub nodes may (totally or partially) fail with a certain probability. \cite{An} propose a model in which two backup hub nodes are determined for each commodity. \cite{Rostami} assume that a finite set of hub breakdown scenarios is known and provide a two-stage formulation for the single allocation hub location problem with possible hub breakdown. \cite{cui2010reliable} provide a stochastic model to determine a subset of the activated hub nodes of given size through which each commodity can be routed, in such a way that if the cheapest route fails, the commodity can be routed through the second cheapest, and so on, or through an emergency facility. The authors provide a Mixed Integer Linear Programming (MILP) formulation for the problem as well as an approximate  Lagrangean relaxation scheme for its resolution based on the ideas in \citep{snyder2005reliability} for the $p$-median problem. The planar version of this model is also analyzed there.\\
\cite{kim2009reliable} propose the reliable $p$-hub location  problem  (PHMR) and the $p$-hub mandatory dispersion (PHMD). In the PHMR the goal is to determine the location of $p$ nodes based on the level of reliability to maximize the completed flows among the set of nodes. The PHMD imposes a certain minimum separation between the $p$ selected hub nodes, and maximizes the reliability of the network.  MILP formulations and heuristic approaches are provided for both problems both in {the} single and {the} multiple-allocation {framework}. Reliability of hub backbone networks has been also studied in \citep{zeng2010reliable,korani2021bi,LiLiShuSongZhang_2021IJOC}.\\
\cite{Parvaresh:2013wc} consider the multiple allocation $p$-hub median problem under intentional disruptions in which the goal is to identify the optimal strategy for the location of $p$ hub nodes by minimizing the expected transportation cost in case the worst-case disruptions is minimized. A bilevel mixed integer formulation is provided as well as a simulated annealing heuristic for its resolution.

The problem of designing robust networks under edges/arcs failures  has also been  studied in the literature under different settings. \cite{aneja2001maximizing} study the single-commodity maximum flow problem for the case when edge failures may occur by means of the maximal residual flow problem, whose goal is to determine the maximal flow in which the largest arc flow is as small as possible. This problem is closely related to the network interdiction problem that consists of determining a certain number of arcs whose removal from the network minimizes the maximum amount of flow that one can send through the network \citep[see e.g.][]{altner2010maximum,cormican1998stochastic,royset2007solving,wood1993deterministic}.
\cite{ma2016minimum} propose the Conditional Value-at-Risk
Constrained Minimum Spanning $k$-Core Problem where the possibility of edge disruptions is prevented when trying to construct minimum cost subgraphs of a network with a minimum number of incident edges at each node.
\cite{andreas2008mathematical} study shortest path problems under link failures by imposing that the probability that all arcs {that} successfully operate in at least one path {greater than certain} threshold value.

However, existing works dealing with potential failure of inter-hub links is very scarce \citep{Mohammadi}. This is precisely the focus of this work, where we introduce the Hub Location Problem under Link Failures (HLPLF), a hub location problem in which activated inter-hub links may fail with a given probability. This can be very useful in typical hub location applications in which the total failure of a hub is highly unlikely, whereas partial failures occur only affecting some of the links incident with the hubs (certain air connections, train lines, etc.) We point out that by protecting inter-hub links under failure we also partially protect hub nodes under {failures}.

For dealing with the HLPLF we propose two alternative models, which guarantee that solution networks are protected under disruption of inter-hub links, in the sense that for each commodity at least one alternative (backup) routing path exists. The main difference between the two models is how backup paths are enforced.

In both cases we consider set-up costs  for both activated hubs and activated inter-hub edges. Thus, we do not fix the number of hubs to activate, {allowing} for incomplete backbone networks. As usual, the routing costs apply a discount factor $\alpha$ to inter-hub arcs. We assume multiple allocation of nodes to activated hubs, although the allocation may be different in the original and backup paths. We further impose that the original routing path of each commodity $r=(o_r, d_r)$ contains exactly one inter-hub arc, which can be a loop.  That is, the original routing path is of the form $(o_r, k, m, d_r)$, where $k$ and $m$ are activated hubs, and $(k, m)$  is an inter-hub arc, which reduces to a loop when $k=m$.  {We are also given the failure probabilities for each potential inter-hub edge. Then, t}he objective is to minimize the sum of the set-up costs of the activated hubs and inter-hub edges, plus the expected routing costs.

Our models can be seen as two-stage stochastic programming models in which the \emph{a priori} solution is determined by the strategic decisions associated with the selection of activated {hub nodes} and inter-hub edges together with an \emph{original plan} given by a set of feasible routing paths, one for each commodity, whereas the recourse action determines a \emph{backup plan}, given by a set of alternative routing paths for the commodities, that can be used in case the inter-hub {edge} of the original plan fails.
As already mentioned, the models that we propose differ on the way backup paths are constructed. The first model imposes that the alternative routing path of each commodity contains exactly one inter-hub arc {(as in the original plan)}, which can be a loop, and builds it explicitly. The second model is more flexible, in the sense that it allows {for} arbitrarily large sequences of inter-hub arcs {to be} used in the alternative routing paths, although such paths are not built explicitly. This is achieved with a set of exponentially many (on the number of nodes of the network) constraints, by imposing that the cut-set of the backbone network contains at least $\lambda$ edges, for a given integer value of $\lambda \geq 2$. We study some properties of both models and propose a MILP in each case. For the second model, since it has exponentially many constraints, we also propose a branch-and-cut solution algorithm.

 Extensive computational experiments have been carried out on a large set of benchmark instances based on the well-known CAB~\citep{OKelly_EJOR87}, AP~\citep{EK96}, and {TR~\citep{tan2007hub}} datasets, for varying settings of the failure probabilities and other cost parameters. The obtained results are summarized and analyzed, comparing the computational performance of each of the models and the effect of the different parameters.
 Managerial insights are derived from the analysis of the characteristics of the solutions produced by each of the models and their comparison. {In particular, we analyze the distribution of the costs among the different elements considered (hubs and links set-up costs and routing costs), the number of activated hub and links, and the density of the obtained backbone networks}.
 Finally, an empirical analysis of the two proposed models has been carried out. We compare solutions obtained with the different models in terms of efficiency and robustness. For this analysis, multiple failure scenarios have been generated from optimal solutions to the underlying deterministic hub location model and their \emph{a posteriori} {capability to re-route the commodities}, tested against that of the proposed models. The obtained results assess the validity of the proposal.

The remainder of this paper is structured as follows. In Section \ref{sect1.1} we introduce the notation that we will use and formally define the HLPLF from a general perspective. Section \ref{sect:unrestricted} is devoted to the first HLPLF model that we study in which it is assumed that the alternative paths for the commodities contain exactly one inter-hub arc. We study some of its properties and propose a MILP formulation for it. The model in which we impose that the backbone network is $\lambda$-connected is studied in Section \ref{sect:unrestricted2} where we also present a MILP formulation for it. Section \ref{sec:comput} describes the computational experiments we have carried out and summarizes the obtained results.
Some managerial insights from the analysis of the structure of the solution networks produced by each of the models are also derived in this section.
Finally,  Section \ref{sec:comput_simul} describes the empirical analysis that has been carried out in which multiple failure scenarios have been generated from optimal solutions to the underlying deterministic hub location model and their \emph{a posteriori} capability of re-routing the commodities tested against that obtained with the proposed models. The paper closes in Section \ref{sec:conclu} with some conclusions.

\section{Notation and definition of the problem}\label{sect1.1}

Consider a graph $N=(V, E)$, where the node set $V=\{1, 2,\dots, n\}$ represents a given set of users and the edge set $E$ the existing connections between pairs of users.  We assume that $N$ is a complete graph and that $E$ contains loops, i.e. for all $i\in V$, edge $\{i,i\}\in E$.
We further assume that potential locations for hubs are placed at nodes of the graph and the set of potential locations coincides with $V$. For each potential location $k\in V$,  we denote by $f_k$ the set-up cost for activating a hub at node $k$.
Any pair of hub nodes can be connected by means of an \emph{inter-hub} edge, provided that both endnodes $k$ and $l$ are activated as hub nodes as well.

The set $E$ will be referred to as {the} set of potential inter-hub edges or just as set of potential hub edges. Activated hub edges  incur set-up costs as well; let  {$h_{kl}\geq0$} be the set-up cost for activating hub edge $\{k, l\}\in E$. A set of activated hubs will be denoted by $H \subseteq V$, and a set of activated hub edges for $H$ by $E_H \subseteq E[H]$, where $E[H]$ is the set of edges with both endnodes in $H$.
   Note that the assumption that $N$ is a complete graph implies no loss of generality, since ($i$) arbitrarily large set-up costs can be associated with nodes that are not potential hubs; and, ($ii$) arbitrarily large activation costs can be associated with non-existing hub edges.
Activating a hub edge allows to send flows through it in either direction. Let $A=\{(i,j) \cup (j,i): \{i, j\}\in E, \, i,j\in V\}$ be the arc set. Arcs in the form $(i,i)$ will also be called \emph{loops} and we will use $A_H\subseteq A$ to denote the set of inter-hub arcs induced by $E_H$.

Service demand is given by a set of commodities defined over pairs of users, indexed in a set $R$. Let $\mathcal{D}=\{(o_r, d_r, w_r): r\in R\}$ denote the set of commodities, where the triplet $(o_r, d_r, w_r)$ indicates that an amount of flow $w_r\ge 0$ must be routed  from origin $o_r\in V$ to destination $d_r\in V$. The origin/destination pair associated with a given commodity will also be referred to as its OD pair. Commodities must be routed via paths of the form $\pi=(o_r, k_1, \ldots, k_{s}, d_r)$ with $k_i\in H$, $1\le i\le s$. Similarly to most Hub Location Problems (HLPs), a routing path $\pi=(o_r, k_1, \ldots, k_{s}, d_r)$ is \emph{feasible} if $(1)$ it includes at least one hub node, i.e. $s\geq 1$, and $(2)$ the underlying edges of all traversed arcs other than the access and delivery arcs, $(o_r, k_1)$ and $(k_{s}, d_r)$, respectively, are activated inter-hub edges, i.e., $\{k_i, k_{i+1}\}\in E_H$, $1\le i\le s-1$.
Note that this implies that in any feasible path all intermediate nodes are activated as hubs as well, i.e. $k_i\in H$, $1\leq i\leq s$. In the following the set of feasible paths for a given commodity $r\in R$  will be denoted by $\Pi_{(H,E_H)}(r)$.

Most HLPs studied in the literature do not consider loops explicitly. Then, only \emph{proper} arcs $(k, l)$ with $k\ne l$ can be considered as hub arcs. In this work we follow a slightly more general setting in which loops of the form $(k, k)$ can also be considered as hub arcs. Then, if a loop $(k, k)$ is used in a routing path, it is required not only that $k$ is activated as a hub node, but also that the loop $\{k, k\}$ is activated as a hub edge as well.

Throughout we assume multiple allocation of commodities to open hubs. That is, it is possible that two commodities with the same origin are routed using a different  access hub. Routing flows through the arcs of a hub-and-spoke network incurs different types of costs. These costs, which may depend on the type of arc, account for transportation costs as well as for some additional collection/handling/distribution costs at the endnodes of the arcs. As usual in the literature, we assume that transportation costs of flows routed through inter-hub arcs are subjected to a discount factor $0\leq \alpha \leq 1$. In this work we will denote by $c_{ij}\geq 0$ the  unit routing cost  through inter-hub arc $(i,j)\in A_H$, which includes discounted transportation costs and handling costs and we will denote by $\bar{c}_{ij}\geq 0$ the unit routing   cost for an access or a delivery arc $(i,j)\in A\setminus A_H$, which could also incorporate  different discounted access or delivery costs.

Thus, with the above notation, the routing cost of commodity $r\in R$  through a feasible path $\pi=(o_r, k_1, \ldots, k_{s}, d_r)\in \Pi_{(H,E_H)}(r)$ is:
$$
C^r_{\pi}=w_r\left(\overline c_{o_rk_1}+\sum_{i=1}^{s-1}c_{k_ik_{i+1}} + \overline c_{k_{s}d_r}\right),
$$
where the first and last addends correspond to the access and delivery arcs, respectively, and the intermediate ones are the service costs through the backbone network  $(H,E_{H})$.

Broadly speaking, under the above assumptions, the goal of a HLP is to decide the location of the hub nodes $H$ and to select a suitable subset of hub edges $E_H$, to  \textit{optimally} route the commodities through the backbone network $(H,E_{H})$ induced by the activated hub nodes and hub edges, so as to minimize the sum of the overall set-up costs for activating hub nodes and hub links, plus the {commodities} routing costs.
With the above notation, this  problem can be stated as:
\begin{align}\label{p0}\tag{\rm HLP}
\min_{H \subseteq V, E_H\subseteq E[H]} \sum_{k \in H} f_k + \sum_{e \in E_H} h_{e}+ \sum_{r \in R}\min_{\pi \in \Pi_{(H,E_H)}(r)} C^r_{\pi}.
\end{align}
Most HLPs studied in the literature (that do not consider loops explicitly) restrict the set of potential paths for routing the commodities to those using at most one hub arc. When loops are also considered as potential hub arcs as we do, the analogous set of potential paths for routing the commodities is restricted to those containing exactly one hub arc. Thus, for $H$ and $E_{H}$ given, the set of potential paths for routing {commodity $r\in R$} is given by $\Pi_{(H,E_H)}(r) =\{\pi=(o_r, k, l, d_r): \, k, l \in H, \, \{k,l\}\in E_{H}\}$. In such a case the routing cost of commodity $r$ through path $\pi=(o_r,k,l,d_r)$  reduces to
$$
C_{kl}^r := C^r_{\pi}=w_r(\overline c_{o_rk}+ c_{kl}+ \overline c_{ld_r}),
$$
\noindent and the HLP simplifies to:
\begin{align}\label{p0_1}\tag{\rm HLP$^1$}
\min_{H \subseteq V, E_{H}\subseteq E[H]} \sum_{k \in H} f_k + \sum_{e \in E_{H}} h_{e} + \sum_{r \in R}\min_{(k,l) \in A_{H}} C^r_{k l}.
\end{align}

{Most hub networks are sensitive to failures in their links, being the impact in some of them particularly harmful for the users. Examples of potential applications of the models that we study include  the management of airlines and airport industries~\citep{campbell2005a}, in which breakdowns in certain flight connections may occur, and passengers are directly affected. Also, in rapid delivery packing systems~\citep{ccetiner2010hubbing}, where the users pay for fast services and failures in the hub network cause large delays.}

 In the remainder of this work we consider HLPs in which the (\emph{original}) routing paths consist of exactly one hub arc (possibly a loop), and we assume that activated hub edges in $E_{H}$ may fail. In case hub edge $\{k,l\}\in E_{H}$ fails, then the inter-hub arcs $(k, l), (l, k)\in A_H$ can no longer be used for routing the commodities. In order to protect solution networks from failure we follow a policy that does not alter the strategic decisions on the activated hubs and inter-hub edges and focuses solely on the operational decisions concerning the re-routing of affected commodities.
 Accordingly,  we impose that, for each commodity, the backbone network $(H,E_{H})$ contains, in addition to the original routing path, some substitute path connecting $o_r$ and $d_r$. Such a substitute path will be referred to as \emph{backup} or \emph{alternative} path.

For each edge $e=\{k,l\} \in E$, let $X_{kl}$ denote the random variable modeling whether $e$ fails. We assume that the random {variable $X_{kl}$ follows a  Bernoulli distribution with probability $p_{kl}$, for each }$\{k,l\}\in E$. When $k\ne l$ failure of edge $e$ arises not only when the link $\{k,l\}$ can no longer be used, but also when, for any reason, the collection and redistribution services at any endnode of edge $\{k,l\}$  cannot be carried out. Thus $p_{kl}$ represents the probability that any of these events happen.
In case edge $e$ is a loop, i.e., $e=\{k,k\}$, with $k\in H$, then $p_{kk}$ represents the probability that the handling process carried out when $k$ is used as the unique intermediate hub fails.

Observe that when edges may fail with a given probability, the costs of feasible routing paths are also random variables, which will be denoted by $\mathcal{C}_r$, $r\in R$. Furthermore, the probability distribution of $\mathcal{C}_r$, $r\in R$, is dictated by the failure probability distribution of the involved inter-hub edges. In particular, when $\pi_0 = (o_r,k,l,d_r)$ is the original routing path of a given commodity $r\in R$, the expected routing cost of commodity $r$ can be calculated as:
$$
E[\mathcal{C}_r| \pi_0] = (1-p_{kl}) C^r_{kl} + p_{kl} C^r_{\overline \pi_0}.
$$
where $\overline \pi_0 \in \Pi_{(H,E_H)}(r) \backslash\{\pi_0\}$ is the backup path  in case $\{k,l\}$ fails.

In the following we deal with the problem of finding hub networks protected against  {inter-hub edge} failures under {the} above assumptions.
For this, in the objective function, instead of considering the costs of the routing paths, we will consider their expected routing cost. That is, the HLPLF can be stated as:
\begin{align}\label{p1}\tag{HLPLF}
\min_{H \subseteq V, E_H\subseteq E[H]} \sum_{k \in H} f_k + \sum_{e \in E_H} h_{e}+ \sum_{r\in R} \min_{\pi_0\in\Pi_{(H,E_H)}(r)} E[\mathcal{C}_r| \pi_0].
\end{align}
Indeed, multiple alternatives fall within the above generic framework which differ from each other in how the alternative routing paths are obtained. In the following sections we propose two alternative models for determining such backup paths, based on different assumptions, and provide mathematical programming formulations for each of them.

\section{HLPLF with single inter-hub arc backup paths}\label{sect:unrestricted}

The \ref{p1} that we address in this section enforces that the alternative paths for routing the commodities have the same structure as the original ones. That is, we assume that backup paths contain at least one hub and have exactly one inter-hub arc {(possibly a loop)}. This avoids having to use many transshipment points in case of failure. This model will be referred to as (HLPLF-1BP).

Given a commodity $r\in R$, the backbone network $(H,E_{H})$ and the original routing path $\pi_0 = (o_r,k,l,d_r)$, we assume that the backup path is in the form $\overline \pi_0 = (o_r,\bar k,\bar l,d_r)$,
with  $\{k, l\}\neq\{\bar k, \bar l\}$.

Thus, the expected routing cost of commodity $r$ is:
$$
E[\mathcal{C}_r|\pi_0] = (1-p_{kl}) C^r_{kl} + p_{kl} C^r_{\bar k \bar l}.
$$

Figure \ref{fig00} illustrates the different situations that may arise in case a hub link fails. Figure \ref{fig0:0} shows a backbone network with four hub nodes ($k$, $l$, $m$ and $q$) and five hub links, two of them corresponding to loops, $\{m,m\}$ and $\{l,l\}$, and rest ones corresponding to edges $\{k,q\}$, $\{k,l\}$ and $\{q,m\}$. The figure also depicts the origin ($o_r$) and destination ($d_r$) of a given commodity $r\in R$ and a possible path for this commodity through hub arc $(k,l)$.  We assume that it is the \emph{original path} for the commodity $r\in R$. Access/distribution arcs are depicted as dashed lines and hub edges as solid lines.

Figures \ref{fig0:a} and \ref{fig0:b} show different single inter-hub arc backup paths for  commodity $r\in R$ in case the original one fails. In Figure \ref{fig0:a}, the backup path uses hub-arc $(q,m)$ to re-route the commodity, while in Figure \ref{fig0:b}  the backup path uses loop arc $(l,l)$.

%\begin{figure}[H]
% \begin{center}
%%\includegraphics[width=0.6\textwidth]{fig1}
%\includegraphics[width=0.5\textwidth]{fig3a}
%\end{center}
%%\caption{Illustration of a backbone network to route a commodity $r=(o_r=i,d_r=j,w_r)$.\label{fig:0}}
%\caption{Illustration of a network with four nodes and one commodity $r=(o_r,d_r,w_r)$. The backbone network has five inter-hub links, two of them loops ($\{l,l\}$ and $\{m,m\}$).\label{fig:0}}
%\end{figure}

\begin{figure}[h]
     \centering
     \begin{subfigure}[b]{0.31\textwidth}
         \centering
         \includegraphics[width=\textwidth]{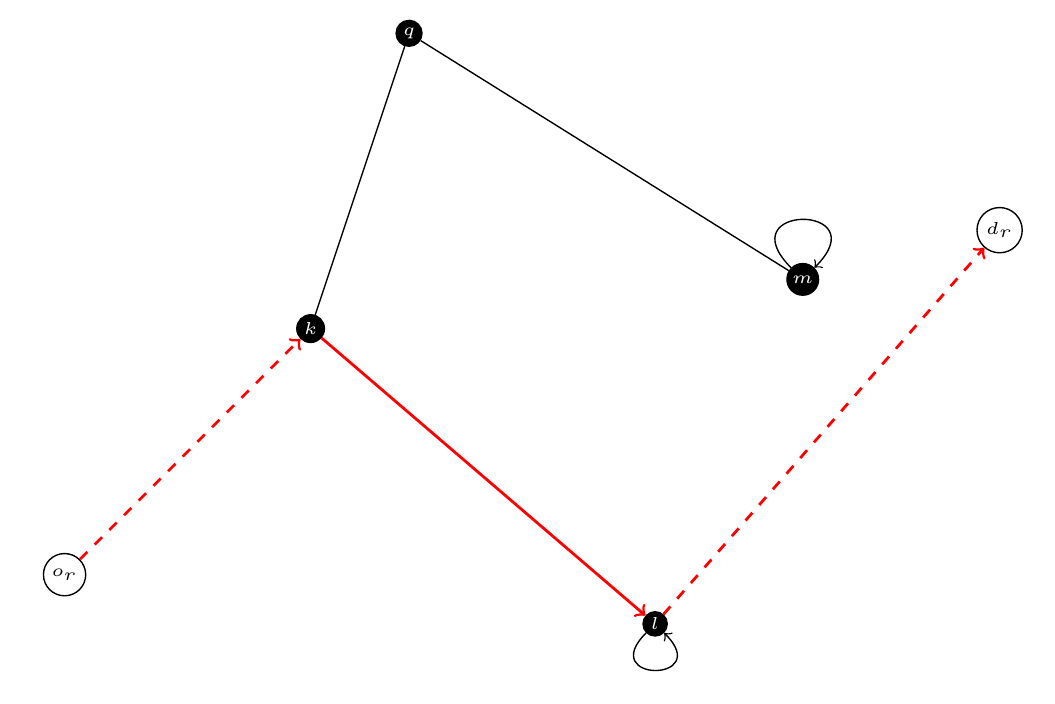}
         \caption{\scriptsize Original path via $(q,\ell)$. \label{fig0:0}}
     \end{subfigure}\quad
     \begin{subfigure}[b]{0.31\textwidth}
         \centering
         \includegraphics[width=\textwidth]{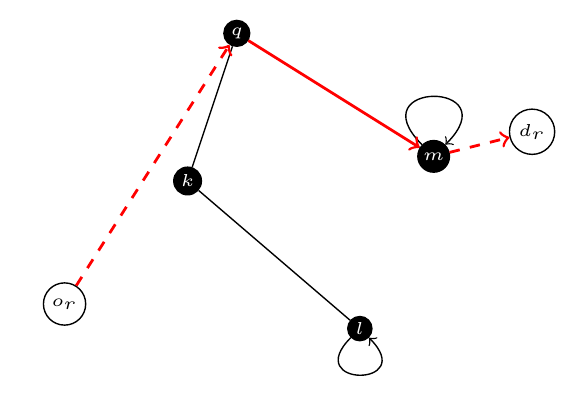}
         \caption{\scriptsize Backup path via $(q,m)$. \label{fig0:a}}
     \end{subfigure}\quad
     \begin{subfigure}[b]{0.31\textwidth}
         \centering
         \includegraphics[width=\textwidth]{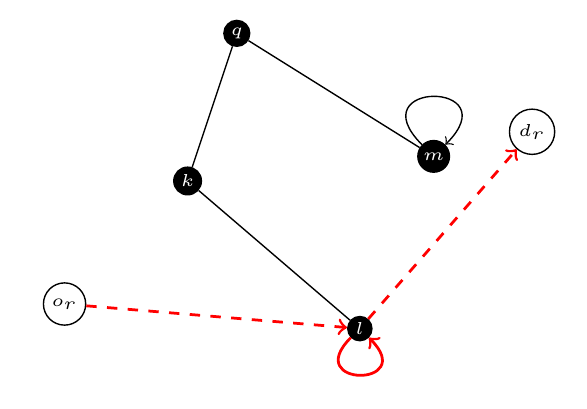}
         \caption{\scriptsize Backup path via $(l,l)$. \label{fig0:b}}
     \end{subfigure}
     %     \begin{subfigure}[b]{0.32\textwidth}
%         \centering
%         \includegraphics[width=\textwidth]{fig3e}
%         \caption{Backup path via hub arc $(q,q)$. \label{fig0:c}}
%     \end{subfigure}
     \caption{A network with four nodes and one commodity $r=(o_r,d_r,w_r)$: Original path, $\pi_0 = (o_r,k,l,d_r)$, and different backup paths in case the original hub arc $(k,l)$ fails. \label{fig00}}
    \end{figure}

Next we develop a mathematical programming formulation for the above problem, first introducing the decision variables.

We  use the following variables associated with the design decisions on the elements of the network that are activated, hubs and edges:
$$
z_k = \left\{\begin{array}{cl}
1 & \mbox{if a hub is opened at the potential hub node $k$,}\\
0 & \mbox{otherwise}
\end{array}\right. \quad \text{ for $k \in V$.}
$$
$$
y_{kl} = \left\{\begin{array}{cl}
1 & \mbox{if hub edge $\{k, l\}$ is activated,}\\
0 & \mbox{otherwise}
\end{array}\right. \quad \text{ for $\{k,l\} \in E$.}
$$

The formulation uses two additional sets of variables, which respectively represent the original and alternative routing path for each commodity. In particular, for $r\in R$ {and} $(k,l) \in A$:
$$
x^r_{kl} = \left\{\begin{array}{cl}
1 & \mbox{ if the original routing path for commodity $r$ is $(o_r,k,l,d_r)$}\\
0 & \mbox{otherwise,}
\end{array}\right.
$$

$$
\bar x^r_{kl} = \left\{\begin{array}{cl}
1 & \mbox{  if the alternative path for commodity $r\in R$ is $(o_r,k,l,d_r)$,}\\
0 & \mbox{otherwise.}
\end{array}\right.
$$

 With these sets of decision variables the expected routing cost of commodity $r\in R$ can be expressed as:
 
$$
\sum_{(k,l)\in A} x_{kl}^r \Big[C_{kl}^r (1-p_{kl}) +  p_{kl}  \sum_{(k^\prime, l^\prime)\in A \setminus\{(k, l)\cup (l,k)\}} C_{k^\prime l^\prime}^r\bar x^r_{k^\prime \, l^\prime}\Big]
$$

\noindent where the two addends in each term of the above expression correspond to the expected routing cost of the original and backup plan of commodity $r$, respectively, both of which only apply if, in the original plan, the commodity is routed through the inter-hub edge corresponding to the term. In particular,
the first addend gives the overall routing cost for commodity $r$ in case the arc of the backbone network used for routing $r$ in the original plan does not fail (multiplied by the probability of not failing). The second term computes the cost of the alternative routing path, multiplied by the probability of failure of the inter-hub edge of the original plan. Observe that in case $(k,l)$ is the arc used initially by {commodity $r\in R$} and $(k^\prime, l^\prime)$ is the backup arc {for $r$}, one obtains the cost $(1-p_{kl}) C^r_{kl} + p_{kl} C^r_{k^\prime l^\prime}$.

Rearranging terms in the above expression one can rewrite the overall  routing cost function for commodity $r$ as:

$$
\sum_{(k,l)\in A}  C_{kl}^r \Big[(1-p_{kl}) x_{kl}^r + \bar x_{kl}^r \sum_{(k^\prime,l^\prime)\in A \setminus\{(k, l)\cup (l,k)\}} p_{k^\prime\, l^\prime} x^r_{k^\prime \, l^\prime}\Big],$$

\noindent where it can be observed that the impact of a given arc $(k,l)\in A$ in the routing cost of commodity $r$ is either $0$ (if it is not used neither in the original nor the alternative path); $(1-p_{kl}) C_{kl}^r$ if it is used in the original path;  or {$C_{kl}^r p_{k^\prime,l^\prime}$} in case arc $(k,l)$ is used in the alternative path and arc ($k^\prime,l^\prime)$ in the original one.

The above decision variables together with this routing cost function lead to the following Integer Nonlinear Programming formulation for  (HLPLF-1BP):

\begin{subequations}
    \makeatletter
        \def\@currentlabel{${\rm HLPLF-1BP}$}
        \makeatother
       \label{HLPLF-1BP}
        \renewcommand{\theequation}{$1.{\arabic{equation}}$}
\begin{align}
\min & \sum_{k \in V} f_k z_k + \sum_{\{k,l\} \in E} h_{kl} y_{kl} &+&\sum_{(k,l)\in A}  C_{kl}^r \Big[(1-p_{kl}) x_{kl}^r + \bar x_{kl}^r \sum_{(k^\prime,l^\prime)\in A \setminus\{(k, l)\cup (l,k)\}} p_{k^\prime\, l^\prime} x^r_{k^\prime \, l^\prime}\Big]\nonumber\\
 \mbox{s.t. }
& \sum_{(k, l)\in A}x^{r}_{kl} =1\, &&  \forall     {r\in R}\label{model1:1}\\
& \sum_{(k, l)\in A} \bar x^{r}_{kl} =1\, &&  \forall     {r\in R}\label{model1:2}\\
&x^{r}_{kl}+x^{r}_{lk}+\bar x^{r}_{kl}+\bar x^{r}_{lk}\le y_{kl}\, \qquad && \forall    r\in R, \{k, l\}\in E, k\ne l\label{model1:3}\\
&x^{r}_{kk}+\bar x^{r}_{kk}\leq y_{kk}\,\qquad  && \forall  r\in R, \{k, k\}\in E \label{model1:4}\\
&y_{kl}\leq z_k\,\qquad && \forall \{k, l\}\in E\label{model1:5}\\
&y_{kl}\leq z_l\,\qquad  && \forall  \{k, l\}\in E, k\ne l\label{model1:6}\\
&x^r_{kk}+\sum_{\substack{l\in V\\l\ne k}} (x^r_{kl}+x^r_{lk})\le z_k  && \forall r, \forall k\in V\label{model1:7}\\
&\bar x^r_{kk}+\sum_{\substack{l\in V\\l\ne k}} (\bar x^r_{kl}+\bar x^r_{lk})\le z_k  &&\forall r, \forall k\in V\label{model1:8}\\
& x^{r}_{kl}, \bar x^{r}_{kl}\in\{0, 1\} && \forall r\in R, (k, l)\in A\label{int_x}\\
& z_k\in\{0, 1\}         &&    \forall  k\in V\label{int_z}\\
& y_{kl}\in\{0, 1\}        &&    \forall  \{k, l\}\in E.\label{int_y}
\end{align}
\end{subequations}
where constraints \eqref{model1:1} and \eqref{model1:2} enforce that each commodity uses exactly one inter-hub arc both in the original and the backup path.  Constraints \eqref{model1:3} and \eqref{model1:4} impose that the original and the backup path do not coincide. These constraints also guarantee that any used inter-hub edge is activated. Constraints \eqref{model1:5} and \eqref{model1:6} ensure that any endnode of an activated inter-hub edge must be activated as a hub. Constraints \eqref{model1:7} and \eqref{model1:8} are valid inequalities, already proposed in \cite{MARIN2006274}, which reinforce the relationship between the routing variables and the hub activation variables. Finally, \eqref{int_x}--\eqref{int_y} are the domains of the decision variables.

\subsection{Linearization of the objective function}

The reader may have observed the non-linearity of the objective function term corresponding to the expected routing cost. As we explain below this term can be suitably linearized by introducing a new auxiliary variable $P^r_{kl}\in \mathbb{R}_+$ associated with each commodity $r\in R$ and each arc $(k, l)\in A$.

Let $P^r_{kl}= \bar x^r_{kl} \sum_{(k^\prime,l^\prime)\in A \setminus\{(k, l)\cup (l,k)\}} p_{k^\prime l^\prime} x^r_{k^\prime l^\prime}$ denote the probability of using inter-hub arc $(k, l)$ in the alternative path of commodity $r$.  Observe that because of the minimization criterion, the nonnegativity of the routing costs, and constraints \eqref{model1:1}-\eqref{model1:4}, the value of $P_{kl}^r$ can be determined by the following set of constraints:
%$$
\begin{align}
P^r_{kl}  &\geq  \sum_{(k^\prime,\,l^\prime)\in A} p_{k^\prime l^\prime} x^r_{k^\prime l^\prime} + (\bar x^r_{kl} -1), && \forall r\in R,\; \forall\, (k,l)\in A, \label{Pkl}\tag{$1.{12}$}\\
 P_{kl}^r & \geq  0, &&\forall r\in R, \; \forall\, (k, l)\in A,\label{Const_last}\tag{$1.{13}$}
\end{align}
and the objective function can be rewritten as:
\begin{align*}
\sum_{k\in V}f_kz_k+ \sum_{\{k,l\}\in E}h_{kl}y_{kl}+ \sum_{r\in R}\sum_{(k,l)\in A}  C^r_{kl}  \Big((1-p_{kl})x^{r}_{kl} + P_{kl}^r\Big).%\label{OF:lin}
\end{align*}
We can also incorporate the following valid inequalities to reinforce our formulation:
\begin{align}
\sum_{(k,l)\in A} P^r_{kl}\leq {\max_{\{k,l\}\in E} p_{kl}}, && \forall r\in R\label{dv_Pkl}.\tag{$1.{14}$}
\end{align}
Therefore, we have the following MILP formulation for the problem:
\begin{align}\label{M1}\tag{HLPLF-1BP}
\min & \sum_{k\in V}f_kz_k+ \sum_{\{k,l\}\in E}h_{kl}y_{kl} + \sum_{r\in R}\sum_{(k,l)\in A}  C^r_{kl}  \Big((1-p_{kl})x^{r}_{kl} + P_{kl}^r\Big)\\
  \mbox{s.t. }&  \eqref{model1:1}-\eqref{dv_Pkl}.\nonumber\nonumber%\label{Const_last}
\end{align}

Below we state some simple optimality conditions that can be used to reduce the set of decision variables.
\begin{prop}
There is an optimal solution to \eqref{M1} such that  $x^r_{lk}=\bar x^r_{lk}=P^r_{lk} = 0$ for all $r\in R$, $(l,k)\in A$, with $C^r_{kl}\leq C^r_{lk}$.
\end{prop}
\begin{proof}
The proof is straightforward. Indeed, the value of any solution with $x^r_{lk}=1$ where $C^r_{kl}< C^r_{lk}$ will improve by changing the direction in which edge $\{k,l\}$ is traversed, i.e., by doing $x^r_{lk}=0$, $x^r_{kl}=1$. When $C^r_{kl}= C^r_{lk}$ the value of the new solution will not change.

The same argument can be applied for setting $\bar x^r_{lk}=0$ and thus, $P^r_{lk}=0$.
\end{proof}
Note that the above result allows one to reduce by one half the number of decision variables.

In practice, it is likely that there are few possible values for failure probabilities, and the edges of the network are clustered in groups such that, within each group, all edges have the same failure probability. We next analyze such situation.
\begin{rmk}[Clustered sets of edges]
Let us assume that the  edges in $E$  are clustered in $K$ groups $E_1, \ldots, E_K$ such that all edges in $E_s$ have the same failure probability $\rho_s \in [0,1]$, for $s=1, \ldots, K$. Then, in the term of the objective function of \eqref{M1} corresponding to the expected cost of the commodities, variables $P_{kl}^r$ can be substituted by a new set of variables as follows. For $r\in R$, $(k, l)\in A_s$,  being $A_s$ the arc set induced by $E_s$, $s=1, \ldots, K$, let $\xi_{kls}^r$ be a binary variable that takes value one if and only if the original route of commodity $r$ uses some hub arc in the $s$-th cluster (with failure probability $\rho_s$) and in the backup route it uses hub arc $(k,l)$.

Then, the expected routing cost of commodity $r\in R$ can be rewritten as:
$$
\dsum_{s=1}^K \Big[(1-\rho_s) \dsum_{(k,l)\in A_s} C_{kl}^r  x_{kl}^r + \rho_s \dsum_{(k,l)\in A} C_{kl}^r  \xi_{kls}^r\Big].
$$
Using similar arguments as for the linearization of variables $P_{kl}^r$, the values of the $\xi_{kls}^r$ variables can be determined by the following sets of constraints:
\begin{align}
&\xi_{kls}^r \geq \dsum_{(k^\prime,l^\prime)\in A_s} x_{k^\prime l^\prime}^r + (\bar x_{kl}^r -1)&& \forall r\in R, (k,l) \in A, s=1, \ldots, K\label{1_xi}\\
&\dsum_{s=1}^K\xi_{kls}^r=\bar x^r_{kl} &&  \forall r\in R, (k,l) \in A\label{last_xi}\\
&\xi_{kls}^r \geq 0  &&\forall r\in R, (k,l) \in A, s=1, \ldots, K.\label{2_xi}
\end{align}

The particular case of one single cluster  ($K=1$) where all edges have the same failure probability, i.e., $p_{kl}=\rho$
 for all $\{k, l\}\in E$, allows to further simplify the above formulation. Now the index $s$ can be dropped from variables $\xi$ and Constraints \eqref{1_xi}-\eqref{2_xi} are no longer needed, as Constraints \eqref{last_xi} reduce to $\xi_{kl}^r=\bar x^r_{kl}$, $(k,l)\in A$. Then, the  expected routing cost of commodity $r\in R$ simplifies to:
 $$
\dsum_{(k,l)\in A} C_{kl}^r  \Big((1-\rho) x_{kl}^r + \rho \overline x_{kl}^r\Big).
$$
\end{rmk}

\section{HLPLF with $\lambda$-connected backbone networks}\label{sect:unrestricted2}

In this section we introduce a different model for the HLPLF, that will be referred to as $\lambda$-connected HLPLF (HLPLF-$\lambda$). Again we make the assumption that the original routing paths contain at least one hub node and exactly one inter-hub arc, although we follow a different modeling approach as for how to protect the backbone network $(H, E_H)$ against potential failures. On the one hand, we extend the set of alternative paths that can be used when hub edges in original paths fail, and allow for any arbitrarily long chain of arcs connecting the OD pair of each commodity, provided that all its intermediate arcs are activated inter-hub arcs. On the other hand, we no longer make explicit the alternative routing paths for the commodities.  Instead,
we impose that the backbone network is $\lambda$-connected, in the sense that it must contain at least $\lambda$ routing paths connecting any pair of activated hubs $k$, $l\, \in H$ with $k\neq l$, where $\lambda\geq 2$ is a given {integer} parameter. This implies that if some hub arc of the original path fails, then the backbone network contains at least $\lambda-1$ alternative paths connecting the activated hubs.  Note that, this forces the backbone network to have at least $\lambda$ activated hub nodes.
The particular case of HLPLF-$\lambda$ with $\lambda=2$, extends the HLPLF-1BP studied in the previous section, as it enforces at least one backup path in the backbone network in addition to the original one, which can be arbitrarily long.

We recall that for any non-empty subset of nodes $S\subset V$, the cutset associated with $S$ is precisely the set of edges connecting $S$ and $V\backslash S$ namely:
$$
\delta(S)=\{\{k, l\} \in E | k \in S, l \in V\setminus S\}.
$$
Observe that  the backbone network $(H, E_H)$ depicted in Figure  \ref{fig0:0} is 2-connected since any cutset  has at least two edges (possibly one of them being a loop).

Let us introduce the following additional notation. For a given indicator vector $\bar y\in\{0, 1\}^{|E|}$:
$$
\bar y(\delta(S))=\sum_{\{k,l\}\in \delta(S)} \bar y_{kl}.
$$
That is $\bar y(\delta(S))$ gives the number of edges in the cutset $\delta(S)$, that are activated relative to vector $\bar y$.
When $\bar y=y$, with $y$ being the vector of hub edge decision variables as defined in HLPLF-1BP, then $y(\delta(S))$ gives precisely the number of inter-hub edges in the cutset $\delta(S)$.

Figure \ref{fig11} shows different choices for backup paths in case the  hub arc $(k, l)$, using in the original path for commodity $r=(o_r,d_r,w_r)$,  fails. Note that while the backup paths drawn in Figures \ref{fig4b} and \ref{fig4d} are also valid for HLPLF-1BP, the backup path shown in Figure \ref{fig4c} uses two inter-hub arcs, thus not being valid for HLPLF-1BP. Similarly to HLPLF-1BP, loops are also counted for the $\lambda$-connectivity,  as they can be used both in original and backup paths. Note also that in case the original path for commodity $r=(o_r,d_r,w_r)$ uses the loop $(m,m)$, the backup paths in Figure \eqref{fig11} are also feasible.

\begin{figure}[h]
     \centering
     \begin{subfigure}[b]{0.32\textwidth}
         \centering
          \includegraphics[width=\textwidth]{fig3b}
         \caption{\scriptsize Backup path via $(q,m)$. \label{fig4b}}
     \end{subfigure}
     \begin{subfigure}[b]{0.32\textwidth}
         \centering
          \includegraphics[width=\textwidth]{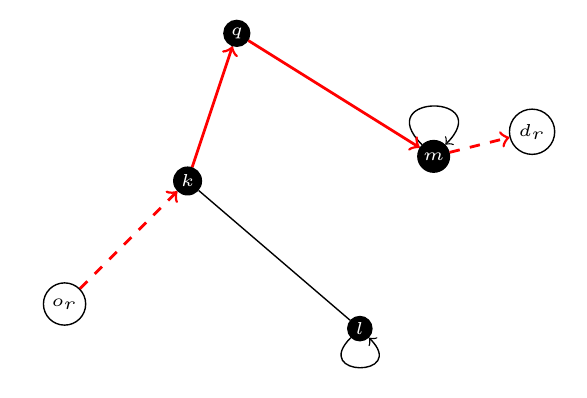}
         \caption{\scriptsize Backup path via $(k,q)$ and $(q,m)$. \label{fig4c}}
     \end{subfigure}
          \begin{subfigure}[b]{0.32\textwidth}
         \centering
          \includegraphics[width=\textwidth]{fig3d}
         \caption{\scriptsize Backup path via  $(l,l)$. \label{fig4d}}
     \end{subfigure}
         \caption{Alternative possibilities for backup paths in the $\lambda$-connected model. \label{fig11}}
    \end{figure}

With the above notation, and taking into account that, by definition, $\delta(S)$ contains no loop hub edges, the $\lambda$-connectivity of the backbone network can be stated by means of the following constraints, associated with each subset $S\subset V$, and each pair of potential hubs $k, l\in V$  with $k\in S$, $l\notin S$:
\begin{align}
y(\delta(S))+y_{kk}\geq \lambda \left(z_{k}+z_{l}-1\right). \label{no-single}
\end{align}
 The right hand side of the above constraint can take a strictly positive value only when $k , l$ are activated hub nodes, that is, when the cutset $\delta(S)$ is a cutset of the backbone network.
In this case, the inequality imposes that $\delta(S)$ contains at least  $\lambda-y_{kk}$ activated hub edges. As indicated, the loop $\{k, k\}$  has also been taken into account  as a potential hub edge since it can be used in  routing paths. Hence, if the loop $\{k, k\}$ is  activated as hub edge, it should be discounted from the number of  hub edges in $\delta(S)$ that must be activated. Summarizing, the above constraint imposes that, if nodes $k, l$, $k\in S$, $l\notin S$, are activated hubs, then the number of hub edges in the cutset $\delta(S)$ must be at least $\lambda-1$  if the loop $\{k, k\} $ is activated  as a hub edge or $\lambda$ otherwise. The $\lambda$-connectivity of singletons can be imposed by means of constraints $y(\delta(k))+y_{kk}\geq \lambda z_k$, $k\in V$, which have an analogous interpretation. 

Below we develop a MILP formulation for the HLPLF-$\lambda$, which incorporates $\lambda$-connectivity by means of the family of constraints \eqref{no-single} introduced above. The formulation uses the same $z$, $y$, and $x$ variables as before. Still, since backup paths are no longer made explicit, variables $\bar x$ used in formulation HLPLF-1BP of the previous section are no longer needed. As explained, the $y$ variables will be used to impose the $\lambda$-connectivity condition, which will be stated by means of an exponential set of constraints.

Given that backup routes are no longer made explicit, we no longer have  closed expressions for their expected routing costs and we must estimate their values. Let us denote by $\bar{C}^r_{kl}$ an estimation of the backup routing cost of commodity $r$ when the hub arc $(k,l)$ of its original routing path fails. The resulting formulation for the HLPLF-$\lambda$ is:

\begin{subequations}
    \makeatletter
        \def\@currentlabel{${\rm HLPLF-}\lambda$}
        \makeatother
       \label{HLPLF-lambda}
        \renewcommand{\theequation}{$2.{\arabic{equation}}$}
\begin{align}
 \min &\sum_{k\in V}f_kz_k+ \sum_{\{k,l\}\in E}h_{kl} y_{kl}   +& \sum_{r\in R}\sum_{(k,l)\in A} & \left[(1-p_{kl}) C^r_{kl}+ p_{kl} \bar{C}^r_{kl}\right] x^{r}_{kl} \nonumber\\% \nonumber \\
 \mbox{s.t. }
& \sum_{(k,l)\in A}x^{r}_{kl} =1\,       \qquad &&  \forall     {r\in R} \label{c1}\\%\nonumber\\
& x^{r}_{kl}+x^{r}_{lk}\le y_{kl}        &&  \forall   r\in R, \{k, l\}\in E, k\neq l \label{c2}\\%\nonumber\\
& x^{r}_{kk}\le y_{kk}         &&  \forall   r\in R, k \in V\label{c3}\\% \nonumber\\
%& x^{r}_{kk}\le y_{kk},         &&  \forall   r\in R, k\in V\label{c4}\\% \nonumber\\
& y_{kl}\leq z_k                                                     && \forall \{k, l\}\in E \label{c5}\\%\nonumber\\
& y_{kl}\leq z_l     && \forall \{k, l\}\in E, k\ne l\label{c6}\\%\nonumber\\
& y(\delta(S))+y_{kk}\geq \lambda  (z_{k}+z_l-1) && \forall S\subset V,\, |S|\geq 2,\,  k\in S,\, l \notin S
\label{cutset}\\
& y(\delta(k)) + y_{kk} \geq \lambda z_k && \forall  k\in V \label{cutset2}\\
& x^{r}_{kl}\in\{0, 1\} && \forall r\in R, (k,l)\in A\label{c7}\\% \nonumber\\
& z_k\in\{0, 1\}         &&    \forall  k\in V\label{c8}\\% \nonumber\\
& y_{kl}\in\{0, 1\}         &&    \forall  \{k, l\}\in E. \label{c9}%\nonumber
\end{align}
\end{subequations}
where constraints \eqref{c1}-\eqref{c6} are similar to \eqref{model1:1}-\eqref{model1:6} but referring to the original path only, and \eqref{cutset} and \eqref{cutset2} are the $\lambda$-connectivity constraints described above.

{Note that once the hub backbone network $(H,E_H)$ is obtained by solving the above problem, one can explicitly compute a backup path for a commodity $r\in R$, whose original path is $(o_r, k_r,l_r, d_r)$, by solving (in polynomial time) a
  shortest path problem from source $o_r$ to destination $d_r$ on the graph $G_r=(V_r,E_r)$ with nodes $V_r=\{o_r, d_r\} \cup H$ and edges $E_r=\{\{o_r, h\}: h \in H\} \cup E_h \cup \{(h,d_r): h\in H\}\backslash \{k_r,l_r\}$.}

\begin{prop}
Let $S\subset V$ be a nonempty subset of nodes with $|S|\leq \lambda-1$. Then, the set of constraints \eqref{cutset} is dominated by the set of constraints, which are also valid for \eqref{HLPLF-lambda}:
\begin{align}\tag{$2.6'$}
&y(\delta(S)) + y_{kk} \geq \lambda z_k && \forall k \in S.\label{singlelambda}
\end{align}
\end{prop}
\begin{proof}
Indeed, \eqref{singlelambda} dominate \eqref{cutset}, since $z_l \leq 1$ implies that $\lambda z_k \ge \lambda (z_k+z_l-1)$.

We now see that \eqref{singlelambda} are valid for \eqref{HLPLF-lambda}. Taking into account that $\lambda$-connectivity with $\lambda\geq 2$ implies that any feasible solution has at least $\lambda$ open hubs and that $|S|\leq \lambda-1$, when $k\in S$ is activated as a hub node (i.e. $z_k=1$), there will be at least one more open hub $\bar l\notin S$ (i.e. $z_{\bar l}=1$). That is, when $z_k=1$ there will be at least one active constraint in the set \eqref{cutset} with right-hand-side value $\lambda(z_k+z_{\bar l}-1)=\lambda z_k$. When $k$ is not activated as a hub node (i.e. $z_k=0$) none of the constraints \eqref{cutset}, nor of the constraints \eqref{singlelambda} will be active. Therefore, the result follows.
\end{proof}

Note that when $S$ is a singleton, i.e., $S=\{k\}$, the set of constraints \eqref{singlelambda} reduces precisely to \eqref{cutset2}. Thus, in what follows, we replace  in \eqref{HLPLF-lambda} both constraints \eqref{cutset} and \eqref{cutset2} by \eqref{singlelambda}.

\subsection{Incorporation of $\lambda$-cutset constraints: a branch-and-cut approach} \label{Algorithm}

As already mentioned, in \eqref{HLPLF-lambda}, the size of the family of constraints \eqref{singlelambda} is exponential in the number of potential hub nodes, $n$. It is thus not possible to solve the formulation  directly with some off-the-shelf solver, even for medium size instances. In this section we present an exact branch-and-cut algorithm for this formulation in which, as usual, the family of constraints of exponential size \eqref{singlelambda} is initially relaxed. The strategy that we describe below is embedded within an enumeration tree and it is applied not only at the root node but also at all explored nodes.
Our separation procedure is an adaptation of the separation procedure for classical connectivity constraints \citep{PG-1985}, and follows the same vein of those applied to more general connectivity inequalities in node and arc routing problems (see, e.g., \citep{BB-COA-1998,AFF09,RPFLBM19} for further details).

The initial formulation includes all constraints \eqref{c1}-\eqref{c6}, and the singleton version of \eqref{singlelambda}. Furthermore, all integrality conditions are relaxed.

Let $(\overline{x},\overline{y} ,\overline{z})$ be the solution to the current LP and let $G(\overline{y})=(V(\overline{y}), E(\overline{y}))$ denote its associated support graph  where  $E(\overline{y})$ consists of all the edges of $E$ such that $\bar y_{kl}>0$ and $V(\overline{y})$ the set of endnodes of the edges of $E(\overline{y})$. Each edge  $(k,l)\in E(\overline{y})$ is associated with a capacity $\bar y_{kl}$.
The separation for inequalities \eqref{singlelambda} is to find $S \subset V$, and $k\in S$, with  $\overline{y}(\delta(S))<\lambda\overline z_{k} - \overline y_{kk}$ or to prove that no such inequality exists. Note that, when they exist, violated $\lambda$-connectivity constraints \eqref{singlelambda} can be identified from a tree of min-cuts associated with $G(\overline{y})$ relative to the capacities vector $\overline{y}$, $T(\overline{y})$.

Therefore, to solve the above separation problem we proceed as follows. For each min-cut $\delta(S)$ of $T(\overline{y})$ of value $\overline y(\delta(S))$, we identify $\overline k \in\arg\max \{\lambda\;\overline z_{k} - \overline y_{kk} : k\in S\}$. Then,  if $\overline y(\delta(S))<\lambda\,\overline z_{\overline k} - \overline y_{\overline k\overline k}$, the inequality \eqref{singlelambda} associated with $S$ and $\overline k$ is violated by $\overline{y}$.

We use the procedure proposed by \cite{G-SJAM-1990} to identify $T(\overline{y})$. Such an algorithm computes $V(\overline{y})$ max-flows in $G(\overline{y})$, so its overall complexity is $\mathcal{O}(|V(\overline{y})|\times |V(\overline{y})|^3)$.

\section{Computational Experience}\label{sec:comput}

In this section we report the results of an extensive battery of computational tests, which have been carried out to analyze the performance of the two modeling approaches for obtaining robust hub networks protected under inter-hub failures, discussed in the previous sections. For the experiments, we have used a large set of benchmark instances based on the well-known CAB (\cite{OKelly_EJOR87}), AP (\cite{EK96}) and {on the} TR (\cite{tan2007hub}) datasets (taken from the \texttt{phub} datasets in ORLIB~\url{http://people.brunel.ac.uk/~mastjjb/jeb/orlib/} and \url{https://ie.bilkent.edu.tr/~bkara/dataset.php}), for varying settings of the failure probabilities and other parameters as described below. All instances were solved with the Gurobi 9.1.1 optimizer, under a Windows 10 environment on an Intel(R)  Core(TM) i7-6700K CPU @ 4.00 GHz 4.01 GHz processor and 32 GB of RAM. Default values were used for all parameters of Gurobi solver and a computing  time limit of 7200 seconds was set.

\subsection{Instances generation} \label{sec:genera}
We have generated several instances based on the entire CAB, AP and TR  datasets with a number of nodes ($n$) initially ranging in $\{10, 15, 20, 25\}$ for the instances based on the CAB and TR datasets and in $\{10, 20, 25\}$ for the instances based on the AP dataset. Let $c^\prime_{kl}$ be the standard unit transportation costs provided in ORLIB for CAB and AP instances or the travel distances provided for the TR instances.

The unit routing costs for access/distribution arcs ($\overline c_{ij}$) and the inter-hub routing costs ($c_{kl}$) have been obtained as follows. We take the original costs as  the unit routing cost through the access and delivery arcs, i.e.,  $\overline c_{ij}=c^\prime_{ij}$. For the routing costs through the inter-hubs arcs, we assume that these costs include  not only transportation costs but also some additional handling costs at the endnodes of the traversed arcs, associated with the collection (at the entering node) and redistribution (at the leaving node) of the routed commodity.  Then, we define the unit routing costs through arc $(k,l) \in A$ as:
$$
c_{kl} = \alpha (a_k +  c^\prime_{kl} + d_l),
$$
where:
\begin{itemize}
\item $\alpha \in [0,1]$ is the usual discount factor applied to  routing costs through inter-hub arcs due to economies of scale. Three values for the discount factor $\alpha\in\{0.2, 0.5, 0.8\}$ have been considered in our study.
\item $a_k\geq 0$ and $d_k\geq 0$ are the unit collection and redistribution costs at node $k$, respectively.    Note that applying the discount factor $\alpha$ to these terms implies no loss of generality. Note also that with this choice of costs, in case  $k=l$, the unit routing (service) cost through the  loop $(k, k)$ reduces to $c_{kk}=\alpha(a_k+d_k)$. In our computational study we define $a_k=d_k=\min\{\min_{j\neq k} c^\prime_{kj}, \min_{j\neq k} c^\prime_{jk}\}$.
\end{itemize}

As usual in the literature~\citep{OKelly_PRS92}, we have considered the same {set-up} costs for all {potential  hubs} $k \in V$, $f_k=100$  for the CAB {dataset},  two types of  set-up costs ($T$ and $L$) for the hub nodes provided with the AP dataset, and the set-up costs provided in the TR dataset. Service demand, $w_r$, $r\in R$,  was also taken from the provided datasets.

As  considered in the literature (see e.g., \cite{alumur2009design}, \cite{calik2009tabu}), the set-up costs for activating hub edges for the CAB and the AP datasets were set:
 $$h_{kl} = \left\{\begin{array}{cl}
100  \dfrac{c_{kl}/\textsc{w}_{kl}}{\textsc{maxw}} & \mbox{if $k\neq l$},\vspace*{0.2cm}\\
100  \dfrac{c_{kl}/\bar{\textsc{w}}}{\textsc{maxw}} & \mbox{if $k=l$}.
\end{array}\right.$$
where \textsc{w} is the normalized vector of flows, $\bar{\textsc{w}}$ is the mean of  \textsc{w} and \textsc{maxw}$=\max\{\frac{c_{ij}}{\textsc{w}_{ij}}\, :\, i,j\in V, \, \textsc{w}_{ij}>0\}$ and for TR those provided in the original dataset.\\

In formulation (HLPLF-$\lambda$),  we have estimated the costs of backup paths as $\bar C_{kl}^r = (1+\beta) C_{kl}^r$ for two different values of $\beta \in \{0.5, 1\}$. Observe that in this case the expected routing cost simplifies to:
$$
\sum_{r\in R}\sum_{(k,l)\in A}  (1+\beta p_{kl}) C^r_{kl}  x^{r}_{kl}.\\
$$

As for the failure probabilities $p_{kl}$, $\{k,l\} \in E$, we have considered three different scenarios:
\begin{itemize}
\item[\textbf{RP:}] Random probabilities. The failure probability of each edge is randomly generated from a uniform distribution, i.e. $p_{kl}\equiv U[0,\rho]$ for all $\{k,l\} \in E$.
\item[\textbf{CP:}] Clustered probabilities. Edges are clustered into three groups, each of them with a different failure probability. For this, each edge $\{k, l\} \in E$ is randomly assigned a failure probability in $p_{kl}\in\{0.1, 0.2, 0.3\}$.
\item[\textbf{SP:}] Same probability. All edges have the same failure probability, i.e. $p_{kl}=\rho$, for all $\{k,l\} \in E$.%, including loops.
\end{itemize}
The values of  the parameter $\rho$ we have used in  \textit{RP} and \textit{SP} scenarios are $\rho \in \{0.1,  0.3\}$. {The files of the randomly generated probabilities are available in the Github repository \url{https://github.com/vblancoOR/HLPLF}.

For each combination of parameters $n,\, \alpha, \,  \rho$, and each {dataset} (CAB, AP  with  type T and L fixed set-up hub nodes costs, and TR)  five different instances have been generated for scenario  of  failure probabilities RP and one instance has been considered for scenario  SP. Five instances have been also generated  for scenario  {CP} and each combination of $n,\, \alpha$, and each {dataset}. Thus,  (HLPLF-1BP), {hereafter called} M1, has been solved on a total of 714 instances.

Concerning formulation (HLPLF-$\lambda$), we considered three different values for the parameter $\lambda$, namely $\lambda=2$, $\lambda=3$  and $\lambda=4$ (we call {the corresponding} models M2\_2, M2\_3 and M2\_{{4}},  respectively). Thus,  (HLPLF-$\lambda$) has been solved on a total of 4284 instances. Additionally, for comparative purposes, we have solved 42 instances of the Uncapacitated Hub Location Problem, in which no protection under failures is considered. This model will be referred to as M0.

Finally, to test the scalability of our formulations, a second experiment was carried out on {a set of} larger instances ($n \in \{ 40, 50\}$)  based on the AP and TR {datasets}  considering only  (HLPLF-$\lambda$), which, as we will see, is the most promising formulation, for $\lambda \in\{2,4\}$ and $\beta=1$. We have solved a total of {612} instances in this second study.
Overall, 5652 instances have been solved.

Table \ref{tab:summary} summarizes the main characteristics of the testing instances and the selected parameters.

\begin{table}%[htbp]
  \begin{center}
\hspace*{-1cm}
		\def\arraystretch{2.2}
\adjustbox{scale=0.8}{ \begin{tabular}{c|c|c|c|c|c|c}
  {Instances} & $n$ & $\alpha$ & $c_{kl}$ & $f_k$& $h_{kl}$ & $\lambda$\\%\vspace{5pt}\\
  \hline
CAB & $\{10, 15, 20, 25\}$ & \multirow{3}{*}{\{0.2, 0.5, 0.8\}} & \multirow{3}{*}{$a_k +  c^\prime_{kl} + d_l$} & 100 & \multirow{2}{*}{$\left\{\begin{array}{cl}
100  \dfrac{c_{kl}/\textsc{w}_{kl}}{\textsc{maxw}} & \mbox{ if } k\neq l,\vspace*{0.2cm}\\
100  \dfrac{c_{kl}/\bar{\textsc{w}}}{\textsc{maxw}} & \mbox{ if } k=l.
\end{array}\right.$} & \multirow{3}{*}{$\{2, 3, 4\}$}\\
AP & $\{10, 20, 25, 40, 50\}$  &           &                                                & \multirow{2}{*}{Data file} & &\\
TR & $\{10, 15, 20, 25, 40, 50\}$  &           &                                                & &Data file &\vspace{7pt}\\
\hline
%\vspace{5.pt}
\hline
\multicolumn{7}{c}{Failure probabilities}\\
\hline
%\multicolumn{3}{c|}{Random probabilities (RP)} & \multicolumn{4}{l}{$\qquad p_{kl}\sim U[0,\rho], \quad\qquad\rho\in\{0.1, 0.3\}$}\\
%\multicolumn{3}{c|}{Clustered probabilities (CP)} & \multicolumn{4}{l}{$\qquad p_{kl}\in\{0.1, 0.2, 0.3\}$}\\
%\multicolumn{3}{c|}{Same probability (SP)} & \multicolumn{4}{l}{$\qquad p_{kl}=\rho, \qquad\quad \qquad\rho\in\{0.1, 0.3\}$}\\
\multicolumn{3}{c|}{Random probabilities (RP)} & \multicolumn{2}{l}{$\qquad p_{kl}\sim U[0,\rho],$} & \multicolumn{1}{l}{$\rho\in\{0.1, 0.3\}$ }&\\
\multicolumn{3}{c|}{Clustered probabilities (CP)} & \multicolumn{2}{l}{$\qquad p_{kl}\in\{0.1, 0.2, 0.3\}$}&\multicolumn{1}{l}{}&\\
\multicolumn{3}{c|}{Same probability (SP)} & \multicolumn{2}{l}{$\qquad p_{kl}=\rho,$} &\multicolumn{1}{l}{ $\rho\in\{0.1, 0.3\}$}&\\
\hline
   \end{tabular}}%
    \vspace*{0.25cm}\\
   \caption{Summary of instances and parameters.\label{tab:summary}}
  \end{center}
\end{table}

\begin{table}%[htbp]
  \begin{center}
  \scriptsize%\centering
\adjustbox{scale=0.8}{\begin{tabular}{ccc|rrrrr|rrrrr|rrrrr}
          &       &       & \multicolumn{5}{c|}{CPUTime}          & \multicolumn{5}{c|}{MIPGAP}           & \multicolumn{5}{c}{\%Solved} \\
          &       &       & \multicolumn{2}{c}{RP} & \multicolumn{1}{c}{\multirow{2}[3]{*}{CP}} & \multicolumn{2}{c|}{SP} & \multicolumn{2}{c}{RP} & \multicolumn{1}{c}{\multirow{2}[3]{*}{CP}} & \multicolumn{2}{c|}{SP} & \multicolumn{2}{c}{RP} & \multicolumn{1}{c}{\multirow{2}[3]{*}{CP}} & \multicolumn{2}{c}{SP} \\
\cline{4-5}\cline{7-10}\cline{12-15}\cline{17-18}     n    & $\alpha$ & Data  & \multicolumn{1}{c}{0.1} & \multicolumn{1}{c}{0.3} &       & \multicolumn{1}{c}{0.1} & \multicolumn{1}{c|}{0.3} & \multicolumn{1}{c}{0.1} & \multicolumn{1}{c}{0.3} &       & \multicolumn{1}{c}{0.1} & \multicolumn{1}{c|}{0.3} & \multicolumn{1}{c}{0.1} & \multicolumn{1}{c}{0.3} &       & \multicolumn{1}{c}{0.1} & \multicolumn{1}{c}{0.3} \\
    \hline
    \multirow{12}[5]{*}{10} & \multirow{4}[2]{*}{0.2} &${\rm AP}_T$& 4     & 7     & 15    & 1     & 1     & 0.00  & 0.00  & 0.00  & 0.00  & 0.00  & 100   & 100   & 100   & 100   & 100 \\
          &       & ${\rm AP}_L$ & 7     & 93    & 13    & 2     & 3     & 0.00  & 0.00  & 0.00  & 0.00  & 0.00  & 100   & 100   & 100   & 100   & 100 \\
          &       & CAB   & 143   & 5198  & \texttt{TL}   & 1     & 1     & 0.00  & 0.45  & 1.08  & 0.00  & 0.00  & 100   & 40    & 0     & 100   & 100 \\
          &       & TR    & 4     & 9     & 44    & 0     & 1     & 0.00  & 0.00  & 0.00  & 0.00  & 0.00  & 100   & 100   & 100   & 100   & 100 \\
\cline{2-18}          & \multirow{4}[2]{*}{0.5} &${\rm AP}_T$& 6     & 21    & 14    & 1     & 1     & 0.00  & 0.00  & 0.00  & 0.00  & 0.00  & 100   & 100   & 100   & 100   & 100 \\
          &       & ${\rm AP}_L$ & 10    & 112   & 11    & 2     & 3     & 0.00  & 0.00  & 0.00  & 0.00  & 0.00  & 100   & 100   & 100   & 100   & 100 \\
          &       & CAB   & 62    & 5391  & 1864  & 1     & 1     & 0.00  & 0.95  & 0.00  & 0.00  & 0.00  & 100   & 40    & 100   & 100   & 100 \\
          &       & TR    & 5     & 11    & 63    & 0     & 1     & 0.00  & 0.00  & 0.00  & 0.00  & 0.00  & 100   & 100   & 100   & 100   & 100 \\
\cline{2-18}          & \multirow{4}[1]{*}{0.8} &${\rm AP}_T$& 7     & 13    & 15    & 1     & 1     & 0.00  & 0.00  & 0.00  & 0.00  & 0.00  & 100   & 100   & 100   & 100   & 100 \\
          &       & ${\rm AP}_L$ & 9     & 99    & 12    & 1     & 3     & 0.00  & 0.00  & 0.00  & 0.00  & 0.00  & 100   & 100   & 100   & 100   & 100 \\
          &       & CAB   & 25    & \texttt{TL}   & 727   & 1     & 1     & 0.00  & 0.61  & 0.00  & 0.00  & 0.00  & 100   & 0     & 100   & 100   & 100 \\
          &       & TR    & 5     & 19    & 53    & 0     & 1     & 0.00  & 0.00  & 0.00  & 0.00  & 0.00  & 100   & 100   & 100   & 100   & 100 \\
    \hline
    \multirow{6}[4]{*}{15} & \multirow{2}[1]{*}{0.2} & CAB   & \texttt{TL}   & \texttt{TL}   & \texttt{TL}   & 2     & 7     & 0.60  & 8.35  & 9.37  & 0.00  & 0.00  & 0     & 0     & 0     & 100   & 100 \\
          &       & TR    & 51    & 71    & 287   & 4     & 6     & 0.00  & 0.00  & 0.00  & 0.00  & 0.00  & 100   & 100   & 100   & 100   & 100 \\
   \cline{2-18}        & \multirow{2}[2]{*}{0.5} & CAB   & 4577  & \texttt{TL}   & \texttt{TL}   & 5     & 8     & 0.37  & 9.85  & 7.62  & 0.00  & 0.00  & 40    & 0     & 0     & 100   & 100 \\
          &       & TR    & 52    & 89    & 619   & 3     & 5     & 0.00  & 0.00  & 0.00  & 0.00  & 0.00  & 100   & 100   & 100   & 100   & 100 \\
   \cline{2-18}        & \multirow{2}[1]{*}{0.8} & CAB   & 511   & \texttt{TL}   & \texttt{TL}   & 4     & 6     & 0.00  & 8.76  & 3.65  & 0.00  & 0.00  & 100   & 0     & 0     & 100   & 100 \\
          &       & TR    & 55 & 138   & 412   & 3     & 4     & 0.00  & 0.00  & 0.00  & 0.00  & 0.00  & 100   & 100   & 100   & 100   & 100 \\
    \hline
    \multirow{12}[4]{*}{20} & \multirow{4}[1]{*}{0.2} &${\rm AP}_T$& 1812  & \texttt{TL}   & 6336  & 43    & 30    & 0.00  & 5.92  & 1.38  & 0.00  & 0.00  & 100   & 0     & 20    & 100   & 100 \\
          &       & ${\rm AP}_L$ & \texttt{TL}   & \texttt{TL}   & \texttt{TL}   & 3443  & 2080  & 14.80 & 20.53 & 17.81 & 0.00  & 0.00  & 0     & 0     & 0     & 100   & 100 \\
          &       & CAB   & \texttt{TL}   & \texttt{TL}   & \texttt{TL}   & 13    & 121   & 3.10  & 14.43 & 15.38 & 0.00  & 0.00  & 0     & 0     & 0     & 100   & 100 \\
          &       & TR    & 360   & 1402  & 4131  & 32    & 32    & 0.00  & 0.00  & 1.07  & 0.00  & 0.00  & 100   & 100   & 80    & 100   & 100 \\
\cline{2-18}          & \multirow{4}[2]{*}{0.5} &${\rm AP}_T$& 1520  & \texttt{TL}   & 6611  & 53    & 34    & 0.00  & 2.92  & 1.29  & 0.00  & 0.00  & 100   & 0     & 20    & 100   & 100 \\
          &       & ${\rm AP}_L$ & \texttt{TL}   & \texttt{TL}   & \texttt{TL}   & 4308  & 1257  & 13.11 & 20.30 & 17.19 & 0.00  & 0.00  & 0     & 0     & 0     & 100   & 100 \\
          &       & CAB   & \texttt{TL}   & \texttt{TL}   & \texttt{TL}   & 20    & 51    & 3.01  & 15.65 & 15.55 & 0.00  & 0.00  & 0     & 0     & 0     & 100   & 100 \\
          &       & TR    & 363   & 1850  & 5497  & 21    & 36    & 0.00  & 0.00  & 3.60  & 0.00  & 0.00  & 100   & 100   & 60    & 100   & 100 \\
\cline{2-18}          & \multirow{4}[1]{*}{0.8} &${\rm AP}_T$& 1693  & 7137  & 5780  & 40    & 33    & 0.00  & 3.82  & 1.61  & 0.00  & 0.00  & 100   & 20    & 40    & 100   & 100 \\
          &       & ${\rm AP}_L$ & \texttt{TL}   & \texttt{TL}   & \texttt{TL}   & 2343  & 1809  & 12.60 & 18.98 & 15.63 & 0.00  & 0.00  & 0     & 0     & 0     & 100   & 100 \\
          &       & CAB   & \texttt{TL}   & \texttt{TL}   & \texttt{TL}   & 12    & 27    & 2.85  & 16.39 & 16.22 & 0.00  & 0.00  & 0     & 0     & 0     & 100   & 100 \\
          &       & TR    & 422   & 1864  & 5625  & 19    & 33    & 0.00  & 0.00  & 0.00  & 0.00  & 0.00  & 100   & 100   & 100   & 100   & 100 \\
 \hline
    \multirow{12}[4]{*}{25} & \multirow{4}[1]{*}{0.2} &${\rm AP}_T$& 4774  & 6835  & 4731  & 128   & 129   & 2.35  & 12.50 & 0.00  & 0.00  & 0.00  & 80    & 20    & 100   & 100   & 100 \\
          &       & ${\rm AP}_L$ & \texttt{TL}   & \texttt{TL}   & \texttt{TL}   & \texttt{TL}   & \texttt{TL}   & 16.58 & 22.81 & 19.76 & 11.46 & 14.21 & 0     & 0     & 0     & 0     & 0 \\
          &       & CAB   & \texttt{OoM} & \texttt{TL}   & \texttt{TL}   & 62    & 3208  & 4.33  & 19.52 & 19.21 & 0.00  & 0.00  & 0     & 0     & 0     & 100   & 100 \\
          &       & TR    & 2332  & \texttt{TL}   & \texttt{TL}   & 119   & 169   & 0.00  & 11.98 & 17.25 & 0.00  & 0.00  & 100   & 0     & 0     & 100   & 100 \\
\cline{2-18}          & \multirow{4}[2]{*}{0.5} &${\rm AP}_T$& 6389  & \texttt{TL}   & 4789  & 129   & 179   & 4.97  & 12.60 & 0.00  & 0.00  & 0.00  & 40    & 0     & 80    & 100   & 100 \\
          &       & ${\rm AP}_L$ & \texttt{TL}   & \texttt{TL}   & \texttt{TL}   & \texttt{TL}   & \texttt{TL}   & 14.72 & 20.82 & 18.12 & 8.58  & 13.39 & 0     & 0     & 0     & 0     & 0 \\
          &       & CAB   & \texttt{OoM} & \texttt{OoM} & \texttt{TL}   & 98    & 2403  & 4.69  & 23.21 & 20.37 & 0.00  & 0.00  & 0     & 0     & 0     & 100   & 100 \\
          &       & TR    & 1935  & \texttt{TL}   & \texttt{TL}   & 82    & 128   & 0.00  & 11.54 & 13.29 & 0.00  & 0.00  & 100   & 0     & 0     & 100   & 100 \\
\cline{2-18}          & \multirow{4}[1]{*}{0.8} &${\rm AP}_T$& 6465  & 6789  & 3490  & 125   & 316   & 5.06  & 10.28 & 0.00  & 0.00  & 0.00  & 40    & 20    & 100   & 100   & 100 \\
          &       & ${\rm AP}_L$ & \texttt{TL}   & \texttt{TL}   & \texttt{TL}   & 2927  & \texttt{TL}   & 13.77 & 20.52 & 17.54 & 0.00  & 11.90 & 0     & 0     & 0     & 100   & 0 \\
          &       & CAB   & \texttt{OoM} & \texttt{OoM} & \texttt{TL}   & 71    & 2740  & 7.07  & 23.99 & 18.90 & 0.00  & 0.00  & 0     & 0     & 0     & 100   & 100 \\
          &       & TR    & 2483  & \texttt{TL}   & \texttt{TL}   & 53    & 80    & 0.00  & 12.68 & 14.23 & 0.00  & 0.00  & 100   & 0     & 0     & 100   & 100 \\
    \end{tabular}}%
    \vspace*{0.25cm}\\
   \caption{Average Results for  (HLPLF-1BP).\label{tab:m1}}
  \end{center}
\end{table}%

\subsection{{Numerical results with (HLPLF-1BP) and  (HLPLF-$\lambda$)}}
The results obtained in our first computational study are summarized in {Tables \ref{tab:m1} and \ref{tab:m2}  for  (HLPLF-1BP) and  (HLPLF-$\lambda$), respectively}.  In both tables,  ``RP'',  ``CP'' and ``SP'' stand for the scenarios with   random failure probabilities (with $\rho= 0.1$ and  $\rho=  0.3$),  clustered failure probabilities and same failure probability (with $\rho= 0.1$ and  $\rho=  0.3$), respectively,  {as described in Section \ref{sec:genera}}.
The values of $n$,  $\alpha$ and  ``Data'' {in both tables} indicate the number of nodes in the network, the {value} for the discount factor applied to the routing cost through inter-hubs arcs, and the dataset that {has} been used to {obtain} the costs and the flows, respectively. ``AP$_T$'' and ``AP$_L$'' refer to AP dataset using type T and type L set-up costs for the hub nodes, respectively. In Table \ref{tab:m1},  for scenarios RP and CP, the information contained in each row refers to average values over the five instances with the corresponding combination of parameters, whereas for scenario SP the values of the entries correspond to the unique instance with this combination of parameters. In Table \ref{tab:m1} the numerical results of (HLPLF-1BP) are summarized in three blocks of columns.
 Block  ``CPUTime'' gives the computing times, in seconds, required to solve the instances, block ``MIPGap'' the percentage MIP gaps returned by Gurobi at termination, and block ``\%Solved'' the percentage of instances solved to proven optimality within the time limit. An entry ``\texttt{TL}'' in the CPUTime block means that the time limit of 7200 seconds was reached in all five instances of the group.
 The ``\texttt{OoM}''  entry indicates that  the flag  ``Out of memory'' was the output of the solver in at least one of the instances in the row, and then, the remaining information of the row refers to  average values over the solved instances only (even if none of these instances could be solved to proven optimality).

Table \ref{tab:m2} is organized in three blocks, $\lambda =2$, $\lambda =3$ and $\lambda=4$, for each of the three considered values of $\lambda$ in (HLPLF-$\lambda$).  We have observed that the value of the parameter $\beta$ does not affect the results and thus, in this table the information contained in each row refers to average values of 10 instances ($5\times 2$ different values of $\beta$)  for RP and CP scenarios and refers to the average values of  the two (different values of $\beta$) instances for SP scenario. Using formulation (HLPLF-$\lambda$), all instances have been solved  to proven optimality for all three considered values of $\lambda$. For this reason,  blocks ``MIPGap'' and ``\%Solved'' {have been omitted} in Table  \ref{tab:m2} since  MIPGap is 0.00 for all the instances and the percentage of solved instances  is always  $100\%$.

In Table \ref{tab:m1} we observe a different performance of (HLPLF-1BP) among the instances corresponding to the different scenarios.  Based on the computing times, MIPGAPs, and percentage of solved instances, scenario SP produces the easiest instances, for  all configurations of parameters{, as expected}. One can observe that, for instances with the same probability, all instances generated from the CAB, AP$_T$, and TR datasets  with up to $n=25$, as well as the instances generated from AP$_L$ with up to $n=20$ have been optimally solved within the time limit.

On the other hand, note that instances based on CAB dataset are more difficult to solve than instances based on TR  and AP datasets. TR based instances are the easiest to solve: all instances with up to $n=15$, $95\%$ for $n=20$ and $35\%$ for $n=25$  have been optimally solved within the time limit. Regarding AP based instances, all instances  with  $n=10$ have been optimally solved within the time limit, although AP$_T$ instances  consumed, in general, less computing time. The difference between AP$_T$ and AP$_L$ instances becomes more evident  for $n>10$, since approximately $50\%$ of the AP$_T$ instances were optimally  solved whereas none of the   AP$_L$  instances with random and clustered  probabilities {(scenarios RP and CP)} was solved to proven optimality within the time limit. As mentioned before, CAB instances are the most difficult ones. For $n=10$, $30\%$ of these instances could not be optimally solved solved within the time limit. This percentage increases up to $90\%$ for $n=20$ and up to $100\%$ for $n=25$. Additionally, for $n=25$,  the execution was stopped due to an Out of Memory flag with  $30\%$ of the CAB instances under the random probabilities (RP)  scenario.

\begin{table}%[htbp]
  \centering \scriptsize
    \adjustbox{scale=0.8}{\begin{tabular}{ccc|rrrrr|rrrrr|rrrrr}
          &       &       & \multicolumn{5}{c|}{$\lambda=2$} & \multicolumn{5}{c|}{$\lambda=3$} & \multicolumn{5}{c}{$\lambda=4$} \\
          &       &       & \multicolumn{2}{c}{RP} &       & \multicolumn{2}{c|}{SP} & \multicolumn{2}{c}{RP} &       & \multicolumn{2}{c|}{SP} & \multicolumn{2}{c}{RP} &       & \multicolumn{2}{c}{SP} \\
\cline{4-5}\cline{7-10}\cline{12-15}\cline{17-18}     n    &  $\alpha$  & Data  & \multicolumn{1}{c}{0.1} & \multicolumn{1}{c}{0.3} & \multicolumn{1}{c}{CP} & \multicolumn{1}{c}{0.1} & \multicolumn{1}{c|}{0.3} & \multicolumn{1}{c}{0.1} & \multicolumn{1}{c}{0.3} & \multicolumn{1}{c}{CP} & \multicolumn{1}{c}{0.1} & \multicolumn{1}{c|}{0.3} & \multicolumn{1}{c}{0.1} & \multicolumn{1}{c}{0.3} & \multicolumn{1}{c}{CP} & \multicolumn{1}{c}{0.1} & \multicolumn{1}{c}{0.3} \\
    \hline
    \multirow{12}[6]{*}{10} & \multirow{4}[2]{*}{0.2} & ${\rm AP}\_T$ & 2     & 2     & 3     & 2     & 2     & 3     & 5     & 5     & 3     & 5     & 3     & 3     & 2     & 3     & 2 \\
          &       & ${\rm AP}\_L$ & 6     & 3     & 6     & 7     & 6     & 2     & 2     & 2     & 2     & 2     & 3     & 4     & 3     & 3     & 5 \\
          &       & CAB   & 0     & 0     & 0     & 0     & 0     & 0     & 0     & 0     & 0     & 0     & 1     & 1     & 1     & 1     & 1 \\
          &       & TR    & 1     & 1     & 1     & 1     & 1     & 5     & 4     & 5     & 6     & 2     & 8     & 7     & 6     & 14    & 9 \\
\cline{2-18}          & \multirow{4}[2]{*}{0.5} & ${\rm AP}\_T$ & 2     & 2     & 3     & 4     & 2     & 3     & 4     & 4     & 4     & 4     & 3     & 3     & 2     & 2     & 2 \\
          &       & ${\rm AP}\_L$ & 6     & 3     & 8     & 5     & 3     & 2     & 2     & 2     & 3     & 2     & 3     & 3     & 2     & 2     & 4 \\
          &       & CAB   & 0     & 0     & 0     & 0     & 0     & 0     & 0     & 1     & 0     & 0     & 3     & 2     & 2     & 4     & 3 \\
          &       & TR    & 3     & 1     & 2     & 4     & 2     & 4     & 4     & 4     & 2     & 3     & 8     & 7     & 6     & 6     & 9 \\
\cline{2-18}          & \multirow{4}[2]{*}{0.8} & ${\rm AP}\_T$ & 2     & 3     & 3     & 3     & 2     & 3     & 3     & 3     & 4     & 2     & 2     & 2     & 3     & 2     & 2 \\
          &       & ${\rm AP}\_L$ & 3     & 3     & 5     & 7     & 2     & 2     & 2     & 2     & 2     & 1     & 2     & 3     & 2     & 2     & 3 \\
          &       & CAB   & 0     & 0     & 0     & 0     & 0     & 1     & 0     & 1     & 1     & 0     & 2     & 2     & 2     & 2     & 2 \\
          &       & TR    & 3     & 3     & 3     & 2     & 4     & 4     & 3     & 3     & 4     & 2     & 7     & 8     & 6     & 7     & 6 \\
    \hline
    \multicolumn{1}{r}{\multirow{6}[6]{*}{15}} & \multirow{2}[2]{*}{0.2} & CAB   & 1     & 1     & 1     & 1     & 1     & 1     & 1     & 1     & 1     & 1     & 1     & 1     & 1     & 1     & 1 \\
          &       & TR    & 9     & 6     & 8     & 6     & 7     & 22    & 17    & 25    & 40    & 17    & 42    & 26    & 31    & 40    & 24 \\
\cline{2-18}          & \multirow{2}[2]{*}{0.5} & CAB   & 1     & 1     & 1     & 1     & 1     & 1     & 1     & 2     & 2     & 1     & 3     & 4     & 11    & 3     & 1 \\
          &       & TR    & 31    & 13    & 19    & 26    & 17    & 18    & 23    & 23    & 19    & 18    & 35    & 36    & 31    & 33    & 24 \\
\cline{2-18}          & \multirow{2}[2]{*}{0.8} & CAB   & 1     & 1     & 1     & 1     & 1     & 3     & 2     & 3     & 3     & 3     & 12    & 10    & 12    & 12    & 10 \\
          &       & TR    & 23    & 18    & 16    & 24    & 15    & 17    & 20    & 17    & 14    & 16    & 26    & 23    & 25    & 25    & 17 \\
    \hline
    \multirow{12}[6]{*}{20} & \multirow{4}[2]{*}{0.2} & ${\rm AP}\_T$ & 56    & 42    & 52    & 58    & 35    & 133   & 93    & 132   & 112   & 88    & 134   & 119   & 95    & 98    & 113 \\
          &       & ${\rm AP}\_L$ & 95    & 69    & 99    & 179   & 49    & 102   & 137   & 90    & 81    & 105   & 111   & 97    & 98    & 163   & 137 \\
          &       & CAB   & 1     & 1     & 1     & 1     & 1     & 1     & 2     & 1     & 1     & 1     & 2     & 2     & 1     & 2     & 1 \\
          &       & TR    & 25    & 26    & 23    & 17    & 15    & 98    & 66    & 77    & 50    & 121   & 139   & 133   & 127   & 130   & 95 \\
\cline{2-18}          & \multirow{4}[2]{*}{0.5} & ${\rm AP}\_T$ & 51    & 55    & 61    & 40    & 42    & 119   & 109   & 134   & 73    & 99    & 114   & 111   & 121   & 201   & 81 \\
          &       & ${\rm AP}\_L$ & 101   & 60    & 52    & 36    & 51    & 101   & 82    & 71    & 90    & 117   & 81    & 140   & 89    & 130   & 78 \\
          &       & CAB   & 2     & 2     & 2     & 1     & 1     & 2     & 2     & 2     & 1     & 1     & 2     & 2     & 4     & 2     & 1 \\
          &       & TR    & 46    & 43    & 44    & 28    & 70    & 91    & 81    & 58    & 154   & 98    & 198   & 122   & 141   & 116   & 150 \\
\cline{2-18}          & \multirow{4}[2]{*}{0.8} & ${\rm AP}\_T$ & 41    & 54    & 49    & 27    & 59    & 106   & 91    & 137   & 112   & 114   & 81    & 146   & 97    & 168   & 105 \\
          &       & ${\rm AP}\_L$ & 44    & 77    & 41    & 46    & 47    & 74    & 109   & 90    & 78    & 81    & 95    & 112   & 102   & 103   & 105 \\
          &       & CAB   & 1     & 2     & 1     & 1     & 1     & 1     & 3     & 1     & 1     & 1     & 27    & 40    & 57    & 21    & 21 \\
          &       & TR    & 49    & 42    & 46    & 64    & 34    & 101   & 82    & 52    & 132   & 76    & 168   & 144   & 117   & 163   & 239 \\
    \hline
    \multirow{12}[5]{*}{25} & \multirow{4}[2]{*}{0.2} & ${\rm AP}\_T$ & 301   & 242   & 238   & 111   & 144   & 179   & 205   & 161   & 182   & 251   & 481   & 550   & 412   & 402   & 523 \\
          &       & ${\rm AP}\_L$ & 205   & 246   & 249   & 228   & 105   & 360   & 434   & 356   & 451   & 289   & 628   & 723   & 704   & 537   & 596 \\
          &       & CAB   & 4     & 4     & 4     & 4     & 4     & 4     & 4     & 4     & 4     & 4     & 4     & 8     & 4     & 4     & 4 \\
          &       & TR    & 316   & 173   & 246   & 204   & 304   & 608   & 545   & 557   & 298   & 486   & 1364  & 1383  & 1018  & 1544  & 922 \\
\cline{2-18}          & \multirow{4}[2]{*}{0.5} & ${\rm AP}\_T$ & 187   & 179   & 165   & 108   & 64    & 163   & 177   & 167   & 151   & 150   & 330   & 386   & 277   & 326   & 421 \\
          &       & ${\rm AP}\_L$ & 178   & 403   & 122   & 329   & 250   & 387   & 341   & 329   & 529   & 338   & 543   & 674   & 490   & 676   & 738 \\
          &       & CAB   & 4     & 4     & 4     & 5     & 4     & 4     & 4     & 5     & 4     & 4     & 135   & 220   & 236   & 71    & 11 \\
          &       & TR    & 458   & 425   & 282   & 410   & 596   & 790   & 749   & 974   & 558   & 1003  & 1795  & 1650  & 1376  & 1016  & 2003 \\
\cline{2-18}          & \multirow{4}[1]{*}{0.8} & ${\rm AP}\_T$ & 180   & 121   & 142   & 165   & 40    & 142   & 158   & 141   & 124   & 139   & 266   & 369   & 338   & 303   & 365 \\
          &       & ${\rm AP}\_L$ & 125   & 87    & 153   & 86    & 130   & 271   & 258   & 296   & 181   & 283   & 589   & 383   & 370   & 414   & 396 \\
          &       & CAB   & 3     & 10    & 11    & 4     & 4     & 4     & 8     & 7     & 4     & 4     & 90    & 255   & 150   & 202   & 131 \\
          &       & TR    & 396   & 421   & 348   & 216   & 499   & 848   & 1250  & 1036  & 606   & 928   & 1870  & 1593  & 1599  & 2062  & 2323 \\
    \end{tabular}}%
\vspace*{0.25cm}\\
   \caption{Average CPU times for  (HLPLF-$\lambda$).}
  \label{tab:m2}
\end{table}%

Comparing Table \ref{tab:m1}  with Table  \ref{tab:m2} we observe that  (HLPLF-$\lambda$) is \textit{notably} easier to solve than  (HLPLF-1BP), which can be explained by its smaller number of decision variables.  The difficulty of (HLPLF-$\lambda$) increases with the value of $\lambda$, as reflected by a decrease in its performance for higher values of this parameter. This could be expected, as (HLPLF-$\lambda$) becomes  more restrictive as the value of $\lambda$ increases.
When $\lambda=2$, the average computing time over all the instances is approximately  64 seconds, being  two seconds for the CAB instances, 70 seconds for the AP$_T$ instances, 89 seconds  for the AP$_L$ instances, and 102 for the TR instances.  Note that, unlike  (HLPLF-1BP), CAB instances are less computationally demanding than AP and TR instances. This behavior was also observed for $\lambda = 3$ and $\lambda = 4$. For $\lambda=4$ the average computing time over all the instances is approximately  218 seconds, being  30 seconds for the CAB instances, 168 seconds for the AP$_T$ instances, 225 seconds  for the AP$_L$ instances, and 438 for the TR instances.
 Observe that the value of $\alpha$ also affects the performance of  (HLPLF-$\lambda$),  instances being more difficult for smaller $\alpha$ values, specially for the AP instances.
We also note  that, unlike  (HLPLF-1BP), with  (HLPLF-$\lambda$) there seem to be no noticeably differences among scenarios.

\begin{table}%[htbp]
  \centering\scriptsize
    \adjustbox{scale=0.8}{\begin{tabular}{cccc|rrrrr|rrrrr|rrrrr}
          &       &       &       & \multicolumn{5}{c|}{CPUTime }         & \multicolumn{5}{c|}{MIPGap }          & \multicolumn{5}{c}{\%Solved} \\
          &       &       &       & \multicolumn{2}{c}{RP} &       & \multicolumn{2}{c|}{SP} & \multicolumn{2}{c}{RP} &       & \multicolumn{2}{c|}{SP} & \multicolumn{2}{c}{RP} &       & \multicolumn{2}{c}{SP} \\
\cline{5-6}\cline{8-11}\cline{13-16}\cline{18-19}          & \multicolumn{1}{l}{ n } &  $\alpha$  & \multicolumn{1}{l|}{Data} & \multicolumn{1}{c}{0.1} & \multicolumn{1}{c}{0.3} & \multicolumn{1}{c}{CP} & \multicolumn{1}{c}{0.1} & \multicolumn{1}{c|}{0.3} & \multicolumn{1}{c}{0.1} & \multicolumn{1}{c}{0.3} & \multicolumn{1}{c}{CP} & \multicolumn{1}{c}{0.1} & \multicolumn{1}{c|}{0.3} & \multicolumn{1}{c}{0.1} & \multicolumn{1}{c}{0.3} & \multicolumn{1}{c}{CP} & \multicolumn{1}{c}{0.1} & \multicolumn{1}{c}{0.3} \\
    \hline
    \multirow{18}[12]{*}{$\lambda=2$} & \multirow{9}[6]{*}{40} & \multirow{3}[2]{*}{0.2} & ${\rm AP}_T$ & 371   & 363   & 382   & 361   & 310   & 0.00  & 0.00  & 0.00  & 0.00  & 0.00  & 100   & 100   & 100   & 100   & 100 \\
          &       &       & ${\rm AP}_L$ & 1608  & 1689  & 5418  & 1181  & 1059  & 0.00  & 0.00  & 4.43  & 0.00  & 0.00  & 100   & 100   & 60    & 100   & 100 \\
          &       &       & TR    & 870   & 775   & 773   & 666   & 537   & 0.00  & 0.00  & 0.00  & 0.00  & 0.00  & 100     & 100     & 100     & 100     & 100 \\
\cline{3-19}          &       & \multirow{3}[2]{*}{0.5} & ${\rm AP}_T$ & 344   & 479   & 313   & 322   & 432   & 0.00  & 0.00  & 0.00  & 0.00  & 0.00  & 100   & 100   & 100   & 100   & 100 \\
          &       &       & ${\rm AP}_L$ & 3579  & 3155  & 2487  & 2602  & 1351  & 0.00  & 3.16  & 0.00  & 0.00  & 0.00  & 100   & 80    & 100   & 100   & 100 \\
          &       &       & TR    & 434   & 600   & 601   & 337   & 414   & 0.00  & 0.00  & 0.00  & 0.00  & 0.00  & 100     & 100     & 100     & 100     & 100 \\
\cline{3-19}          &       & \multirow{3}[2]{*}{0.8} & ${\rm AP}_T$ & 331   & 376   & 325   & 303   & 337   & 0.00  & 0.00  & 0.00  & 0.00  & 0.00  & 100   & 100   & 100   & 100   & 100 \\
          &       &       & ${\rm AP}_L$ & 2067  & 1663  & 2126  & 1604  & 692   & 0.00  & 0.00  & 0.00  & 0.00  & 0.00  & 100   & 100   & 100   & 100   & 100 \\
          &       &       & TR    & 499   & 489   & 430   & 512   & 390   & 0.00  & 0.00  & 0.00  & 0.00  & 0.00  & 100     & 100     & 100     & 100     & 100 \\
\cline{2-19}          & \multirow{9}[6]{*}{50} & \multirow{3}[2]{*}{0.2} & ${\rm AP}_T$ & 3986  & 3470  & 5128  & 4147  & 3651  & 0.00  & 0.00  & 2.46  & 0.00  & 0.00  & 100   & 100   & 80    & 100   & 100 \\
          &       &       & ${\rm AP}_L$ & 6427  & 5362  & 6472  & 3273  & \texttt{TL}   & 13.79 & 6.56  & 12.40 & 0.00  & 15.64 & 20    & 60    & 40    & 100   & 0 \\
          &       &       & TR    & 2877  & 2520  & 3385  & 2180  & 3817  & 0.00  & 0.00  & 0.00  & 0.00  & 0.00  & 100     & 100     & 100     & 100     & 100 \\
\cline{3-19}          &       & \multirow{3}[2]{*}{0.5} & ${\rm AP}_T$ & 2635  & 3789  & 4319  & 4823  & 1590  & 0.00  & 0.00  & 2.02  & 0.00  & 0.00  & 100   & 100   & 80    & 100   & 100 \\
          &       &       & ${\rm AP}_L$ & 3819  & 3406  & \texttt{TL}   & 2761  & 2131  & 0.00  & 3.45  & 17.09 & 0.00  & 0.00  & 100   & 80    & 0     & 100   & 100 \\
          &       &       & TR    & 1296  & 3037  & 3264  & 1201  & 1349  & 0.00  & 0.00  & 0.00  & 0.00  & 0.00  & 100     & 100     & 100     & 100     & 100 \\
\cline{3-19}          &       & \multirow{3}[2]{*}{0.8} & ${\rm AP}_T$ & 2862  & 3732  & 2836  & 3736  & 2474  & 0.00  & 0.00  & 0.00  & 0.00  & 0.00  & 100   & 100   & 100   & 100   & 100 \\
          &       &       & ${\rm AP}_L$ & 5872  & 5151  & 7183  & 6279  & 1172  & 9.93  & 2.99  & 14.02 & 0.00  & 0.00  & 40    & 80    & 20    & 100   & 100 \\
          &       &       & TR    & 1010  & 1495  & 1356  & 852   & 1050  & 0.00  & 0.00  & 0.00  & 0.00  & 0.00  & 100     & 100     & 100     & 100     & 100 \\
    \hline
    \multirow{18}[12]{*}{$\lambda=4$} & \multirow{9}[6]{*}{40} & \multirow{3}[2]{*}{0.2} & ${\rm AP}_T$ & 2919  & 2907  & 2300  & 2582  & 6060  & 0.00  & 0.00  & 0.00  & 0.00  & 0.00  & 100   & 100   & 100   & 100   & 100 \\
          &       &       & ${\rm AP}_L$ & \texttt{TL}   & \texttt{TL}   & 5643  & 5427  & 5674  & 49.57 & 50.23 & 19.51 & 0.00  & 0.00  & 0     & 0     & 60    & 100   & 100 \\
          &       &       & TR    & 6958  & 7044  & 6002  & \texttt{TL}   & \texttt{TL}   & 13.64 & 9.16  & 5.89  & 7.14  & 10.61 & 20   & 20   & 60   & 0     & 0 \\
\cline{3-19}          &       & \multirow{3}[2]{*}{0.5} & ${\rm AP}_T$ & 2807  & 2335  & 2293  & 2186  & 2939  & 0.00  & 0.00  & 0.00  & 0.00  & 0.00  & 100   & 100   & 100   & 100   & 100 \\
          &       &       & ${\rm AP}_L$ & 5589  & 4865  & 3875  & 5535  & 2960  & 10.03 & 10.03 & 0.00  & 0.00  & 0.00  & 80    & 80    & 100   & 100   & 100 \\
          &       &       & TR    & \texttt{TL}   & 7030  & 7024  & \texttt{TL}   & \texttt{TL}   & 20.15 & 14.00 & 12.62 & 22.24 & 13.99 & 0     & 20   & 20   & 0     & 0 \\
\cline{3-19}          &       & \multirow{3}[2]{*}{0.8} & ${\rm AP}_T$ & 2315  & 2400  & 2272  & 1640  & 1750  & 0.00  & 0.00  & 0.00  & 0.00  & 0.00  & 100   & 100   & 100   & 100   & 100 \\
          &       &       & ${\rm AP}_L$ & 4959  & 5691  & 3512  & 6253  & 2767  & 19.83 & 9.12  & 0.00  & 0.00  & 0.00  & 60    & 80    & 100   & 100   & 100 \\
          &       &       & TR    & \texttt{TL}   & 6991  & 7139  & \texttt{TL}   & \texttt{TL}   & 17.73 & 13.56 & 12.25 & 17.51 & 14.20 & 0     & 20   & 20   & 0     & 0 \\
\cline{2-19}          & \multirow{9}[6]{*}{50} & \multirow{3}[2]{*}{0.2} & ${\rm AP}_T$ & 7089  & 7198  & 6966  & 7134  & 6004  & 30.03 & 40.96 & 27.71 & 0.00  & 0.00  & 40    & 20    & 40    & 100   & 100 \\
          &       &       & ${\rm AP}_L$ & \texttt{TL}   & \texttt{TL}   & \texttt{TL}   & \texttt{TL}   & \texttt{TL}   & 49.99 & 49.23 & 50.31 & 48.24 & 46.73 & 0     & 0     & 0     & 0     & 0 \\
          &       &       & TR    & \texttt{TL}   & \texttt{TL}   & \texttt{TL}   & \texttt{TL}   & \texttt{TL}   & 17.12 & 13.76 & 16.23 & 17.77 & 16.54 & 0     & 0     & 0     & 0     & 0 \\
\cline{3-19}          &       & \multirow{3}[2]{*}{0.5} & ${\rm AP}_T$ & 6482  & 6995  & 6711  & 5310  & \texttt{TL}   & 20.67 & 39.79 & 25.85 & 0.00  & 49.76 & 60    & 20    & 40    & 100   & 0 \\
          &       &       & ${\rm AP}_L$ & \texttt{TL}   & \texttt{TL}   & \texttt{TL}   & \texttt{TL}   & \texttt{TL}   & 49.68 & 48.38 & 50.03 & 48.53 & 48.19 & 0     & 0     & 0     & 0     & 0 \\
          &       &       & TR    & \texttt{TL}   & \texttt{TL}   & \texttt{TL}   & \texttt{TL}   & \texttt{TL}   & 24.59 & 16.95 & 17.05 & 22.29 & 20.88 & 0     & 0     & 0     & 0     & 0 \\
\cline{3-19}          &       & \multirow{3}[2]{*}{0.8} & ${\rm AP}_T$ & 6933  & 7081  & 6259  & 7202  & 7201  & 28.66 & 39.67 & 0.00  & 29.26 & 46.94 & 40    & 20    & 100   & 0     & 0 \\
          &       &       & ${\rm AP}_L$ & \texttt{TL}   & \texttt{TL}   & \texttt{TL}   & \texttt{TL}   & \texttt{TL}   & 48.91 & 47.81 & 47.93 & 48.24 & 46.94 & 0     & 0     & 0     & 0     & 0 \\
          &       &       & TR    & \texttt{TL}   & \texttt{TL}   & \texttt{TL}   & \texttt{TL}   & \texttt{TL}   & 25.94 & 21.03 & 22.78 & 27.43 & 23.51 & 0     & 0     & 0     & 0     & 0 \\
  %  \bottomrule
    \end{tabular}}%
    \vspace*{0.25cm}
  \caption{Average results for  (HLPLF-$\lambda$) for $n\geq 40$.}
  \label{tab:m240}
\end{table}%

Finally, Table \ref{tab:m240} summarizes the results of our second {set of computational experiments}, which {was} carried out considering only   (HLPLF-$\lambda$)   for $\lambda \in\{2,4\}$) and $\beta=1$ on larger instances ($n \in \{ 40, 50\}$)  based on the AP and TR {datasets}. {In this second set of experiments we did not consider  (HLPLF-1BP) as most instances could not be optimally solved within the time limit already for $n=25$.}

 We can observe that, for $\lambda=2$, all the TR instances, 99\% of  the ${\rm AP}_T$ instances, and 80\% of  the $AP_L $ instances have been solved {to proven optimality} within the time limit, whereas for $\lambda=4$  the  percentage of solved instances was $70\%$ of  the ${\rm AP}_T$ instances, 40\% of  the $AP_L $ instances,  and 6\% of the TR instances. This shows that  (HLPLF-$\lambda$) is able to solve {larger} instances with up to $n=50$, {even if  instances become more challenging as  the value of the parameter $\lambda$  increases.}

\subsection{Managerial Insight \label{sec:comput_manag}}
In this section we derive some managerial insight from the results obtained in our first set of experiments, i.e., $n\leq 25$ when the instances were solved with both formulations, as well from the solutions of these instances for M0 (the uncapacitated HLP with no protection under failures). Figure \ref{fig:cost}  shows {the percentage contribution to the objective function value of the different types of costs:} routing costs, {set-up} costs for activating hubs (Hubs\_Costs) and {set-up} costs for activating inter-hub edges (Links\_Costs).

We have observed that results are similar for AP$_L$ and AP$_T$ data sets and thus, for each {formulation,} we differentiate between datasets CAB, AP and TR, as well as {among} the three values of the $\alpha$ parameter.

\begin{figure}[h]
%\begin{center}
%\hspace*{-0.5cm}
\centering\includegraphics[scale=0.55]{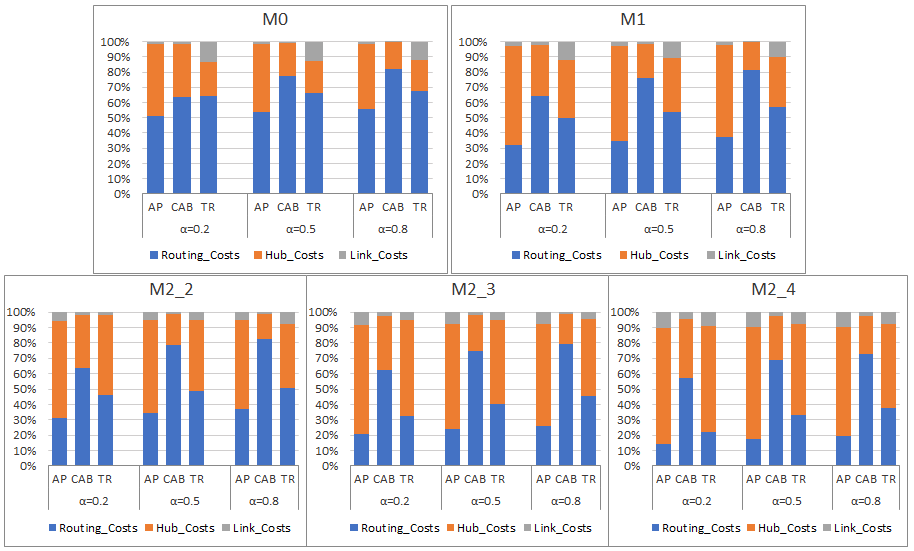}
%\caption{Costs distributions \label{fig:cost}}
\caption{{Percent contribution to the objective function of the different types of costs}
 \label{fig:cost}}
%\end{center}
\end{figure}

We can observe that the percent contribution of the set-up costs for activating inter-hubs edges varies from $0.1\%$ for CAB instances to $13\%$ for TR instances in M0. The {percent contribution} of  {hub set-up} costs {depends},  as expected, on the value of the parameter $\alpha$ but mainly on the dataset and on the model.  For the instances based on the CAB dataset, the percent contribution of hub set-up costs varies from  $20\%$  with  M0,  M1, and  M2\_2  {for}  $\alpha=0.8$, to  $40\%$ with  M2\_4 {for} $\alpha=0.2$. For the instances based on the AP {dataset,  the percent contribution of hub set-up  costs varies from $45\%$ with M0 and  $\alpha\in\{0.5, 0.8\}$ to $75\%$ with M2\_4. Regarding the instances based on the TR dataset, the percent contribution of hub set-up  costs varies from $21\%$ with M0 and  $\alpha\in\{0.5, 0.8\}$ to $69\%$ with M2\_4 for $\alpha=0.2$. The {percent contribution} of  routing costs  also {depends} on the value of the parameter $\alpha$, on the dataset, and on the model.
For the instances based on the CAB dataset, {this percentage  {varies} from  $55\%$ to  $80\%$, {with} the highest values  {for} M0, M1, and M2\_2  {and} $\alpha=0.8$.
For the instances based on the AP {dataset, the percent contribution of} routing costs {varies} from  $15\%$ to  $55\%$, {corresponding} the highest values  {to} M0 with $\alpha=0.8$. Finally,
for the instances based on the TR dataset, the percent contribution of routing costs varies from $22\%$ to 68\%, corresponding again the highest values to M0 with $\alpha=0.8$.
\begin{table}[htbp]
   \scriptsize\centering
    \adjustbox{scale=0.8}{\begin{tabular}{ccc|rrr|rrr|rrr|rrr|rrr}
          &       &       & \multicolumn{3}{c|}{M0} & \multicolumn{3}{c|}{M1} & \multicolumn{3}{c|}{M2\_2} & \multicolumn{3}{c|}{M2\_3} & \multicolumn{3}{c}{M2\_4} \\
     n    &  $\alpha$  & Data  & \multicolumn{1}{l}{\# H} & \multicolumn{1}{l}{\# Lk} & \multicolumn{1}{l|}{\# Lp} & \multicolumn{1}{l}{\# H} & \multicolumn{1}{l}{\# Lk} & \multicolumn{1}{l|}{\# Lp} & \multicolumn{1}{l}{\# H} & \multicolumn{1}{l}{\# Lk} & \multicolumn{1}{l|}{\# Lp} & \multicolumn{1}{l}{\# H} & \multicolumn{1}{l}{\# Lk} & \multicolumn{1}{l|}{\# Lp} & \multicolumn{1}{l}{\# H} & \multicolumn{1}{l}{\# Lk} & \multicolumn{1}{l}{\# Lp} \\
   \hline
    \multirow{9}[6]{*}{10} & \multirow{3}[2]{*}{0.2} & AP    & 1.00  & 1.00  & 1.00  & 2.00  & 2.00  & 2.00  & 2.00  & 3.00  & 2.00  & 3.00  & 6.00  & 3.00  & 4.00  & 10.00 & 4.00 \\
          &       & CAB   & 2.00  & 3.00  & 2.00  & 2.94  & 5.47  & 2.59  & 2.53  & 4.12  & 2.12  & 3.00  & 6.00  & 3.00  & 4.00  & 10.00 & 4.00 \\
          &       & TR    & 1.00  & 1.00  & 1.00  & 2.00  & 2.00  & 1.00  & 3.00  & 3.06  & 0.06  & 4.00  & 6.00  & 0.00  & 5.00  & 10.00 & 0.00 \\
\cline{2-18}          & \multirow{3}[2]{*}{0.5} & AP    & 1.00  & 1.00  & 1.00  & 2.00  & 2.00  & 2.00  & 2.00  & 3.00  & 2.00  & 3.00  & 6.00  & 3.00  & 4.00  & 10.00 & 4.00 \\
          &       & CAB   & 2.00  & 3.00  & 2.00  & 2.35  & 3.71  & 2.35  & 2.03  & 3.09  & 2.03  & 3.00  & 6.00  & 3.00  & 4.00  & 10.00 & 4.00 \\
          &       & TR    & 1.00  & 1.00  & 1.00  & 2.00  & 2.00  & 1.00  & 3.00  & 3.00  & 0.00  & 4.00  & 6.00  & 0.00  & 5.00  & 10.00 & 0.00 \\
\cline{2-18}          & \multirow{3}[2]{*}{0.8} & AP    & 1.00  & 1.00  & 1.00  & 2.00  & 2.00  & 2.00  & 2.00  & 3.00  & 2.00  & 3.00  & 6.00  & 3.00  & 4.00  & 10.00 & 4.00 \\
          &       & CAB   & 2.00  & 2.00  & 2.00  & 2.29  & 2.41  & 2.29  & 2.00  & 3.00  & 2.00  & 3.00  & 6.00  & 3.00  & 4.00  & 10.00 & 4.00 \\
          &       & TR    & 1.00  & 1.00  & 1.00  & 2.00  & 2.00  & 1.00  & 3.00  & 3.65  & 0.65  & 4.00  & 6.12  & 0.12  & 5.00  & 10.00 & 0.00 \\
   \hline
    \multirow{6}[6]{*}{15} & \multirow{2}[2]{*}{0.2} & CAB   & 4.00  & 6.00  & 2.00  & 4.18  & 8.06  & 2.94  & 4.24  & 7.56  & 2.71  & 4.21  & 8.59  & 2.88  & 4.71  & 11.41 & 3.29 \\
          &       & TR    & 1.00  & 1.00  & 1.00  & 2.00  & 2.00  & 1.00  & 3.00  & 3.53  & 0.53  & 4.00  & 6.00  & 0.00  & 5.00  & 10.00 & 0.00 \\
\cline{2-18}          & \multirow{2}[2]{*}{0.5} & CAB   & 2.00  & 3.00  & 2.00  & 2.71  & 4.47  & 2.71  & 2.15  & 3.29  & 2.03  & 3.03  & 6.06  & 3.00  & 4.00  & 10.00 & 4.00 \\
          &       & TR    & 1.00  & 1.00  & 1.00  & 2.00  & 2.00  & 1.00  & 3.00  & 3.12  & 0.12  & 4.00  & 6.00  & 0.00  & 5.00  & 10.00 & 0.00 \\
\cline{2-18}          & \multirow{2}[2]{*}{0.8} & CAB   & 2.00  & 2.00  & 2.00  & 2.18  & 2.65  & 2.18  & 2.00  & 3.00  & 2.00  & 3.00  & 6.00  & 3.00  & 4.00  & 10.00 & 4.00 \\
          &       & TR    & 1.00  & 1.00  & 1.00  & 2.00  & 2.00  & 1.00  & 3.00  & 3.41  & 0.41  & 4.00  & 6.00  & 0.00  & 5.00  & 10.00 & 0.00 \\
   \hline
    \multirow{9}[6]{*}{20} & \multirow{3}[2]{*}{0.2} & AP    & 1.00  & 1.00  & 1.00  & 2.00  & 2.26  & 1.50  & 2.00  & 3.00  & 2.00  & 3.00  & 6.00  & 3.00  & 4.00  & 10.00 & 4.00 \\
          &       & CAB   & 5.00  & 13.00 & 5.00  & 4.29  & 10.88 & 4.24  & 4.76  & 12.12 & 4.59  & 4.76  & 12.29 & 4.74  & 4.74  & 12.94 & 4.74 \\
          &       & TR    & 1.00  & 1.00  & 1.00  & 2.00  & 2.00  & 1.00  & 3.00  & 3.00  & 0.00  & 4.00  & 6.00  & 0.00  & 5.00  & 10.00 & 0.00 \\
\cline{2-18}          & \multirow{3}[2]{*}{0.5} & AP    & 1.00  & 1.00  & 1.00  & 2.00  & 2.21  & 2.00  & 2.00  & 3.00  & 2.00  & 3.00  & 6.00  & 3.00  & 4.00  & 10.00 & 4.00 \\
          &       & CAB   & 4.00  & 9.00  & 4.00  & 3.76  & 8.94  & 3.76  & 3.76  & 8.47  & 3.76  & 3.76  & 8.53  & 3.76  & 4.00  & 10.00 & 4.00 \\
          &       & TR    & 1.00  & 1.00  & 1.00  & 2.00  & 2.00  & 1.00  & 3.00  & 3.41  & 0.41  & 4.00  & 6.06  & 0.06  & 5.00  & 10.00 & 0.00 \\
\cline{2-18}          & \multirow{3}[2]{*}{0.8} & AP    & 1.00  & 1.00  & 1.00  & 2.00  & 2.12  & 2.00  & 2.00  & 3.00  & 2.00  & 3.00  & 6.00  & 3.00  & 4.00  & 10.00 & 4.00 \\
          &       & CAB   & 3.00  & 5.00  & 3.00  & 3.18  & 5.41  & 3.18  & 2.97  & 4.94  & 2.97  & 3.00  & 6.00  & 3.00  & 4.00  & 10.00 & 4.00 \\
          &       & TR    & 1.00  & 1.00  & 1.00  & 2.00  & 2.00  & 1.00  & 3.00  & 3.76  & 0.76  & 4.00  & 6.06  & 0.06  & 5.00  & 10.06 & 0.06 \\
   \hline
    \multirow{9}[5]{*}{25} & \multirow{3}[2]{*}{0.2} & AP    & 1.00  & 1.00  & 1.00  & 2.00  & 2.18  & 1.82  & 2.00  & 3.00  & 2.00  & 3.00  & 6.00  & 3.00  & 4.00  & 10.00 & 4.00 \\
          &       & CAB   & 4.00  & 10.00 & 4.00  & 3.75  & 8.56  & 3.75  & 4.03  & 9.88  & 4.03  & 4.03  & 9.88  & 4.03  & 4.03  & 10.12 & 4.03 \\
          &       & TR    & 1.00  & 1.00  & 1.00  & 2.06  & 2.06  & 0.94  & 3.00  & 3.76  & 0.76  & 4.00  & 6.00  & 0.00  & 5.00  & 10.00 & 0.00 \\
\cline{2-18}          & \multirow{3}[2]{*}{0.5} & AP    & 1.00  & 1.00  & 1.00  & 2.00  & 2.18  & 1.94  & 2.00  & 3.00  & 2.00  & 3.00  & 6.00  & 3.00  & 4.00  & 10.00 & 4.00 \\
          &       & CAB   & 3.00  & 6.00  & 3.00  & 3.62  & 7.62  & 3.62  & 3.00  & 5.94  & 3.00  & 3.00  & 6.00  & 3.00  & 4.00  & 10.00 & 4.00 \\
          &       & TR    & 1.00  & 1.00  & 1.00  & 2.00  & 2.00  & 1.00  & 3.00  & 4.00  & 1.00  & 4.00  & 6.65  & 0.65  & 5.00  & 11.00 & 1.00 \\
\cline{2-18}          & \multirow{3}[1]{*}{0.8} & AP    & 1.00  & 1.00  & 1.00  & 2.00  & 2.15  & 1.97  & 2.00  & 3.00  & 2.00  & 3.00  & 6.00  & 3.00  & 4.00  & 10.00 & 4.00 \\
          &       & CAB   & 3.00  & 5.00  & 3.00  & 3.46  & 6.15  & 3.46  & 2.94  & 4.91  & 2.94  & 3.00  & 6.00  & 3.00  & 4.00  & 10.00 & 4.00 \\
          &       & TR    & 1.00  & 1.00  & 1.00  & 2.00  & 2.00  & 1.00  & 3.00  & 4.00  & 1.00  & 4.00  & 7.00  & 1.00  & 5.00  & 11.00 & 1.00 \\
    \end{tabular}}%
    \vspace*{0.25cm}
   \caption{Average number of open hubs, links and loops}
        \label{tab:nhub}%
\end{table}%

%\newpage

Information about the structure of the optimal backbone network can be found in Table \ref{tab:nhub}, {and Figures \ref{fig:densCAB}, \ref{fig:densAP} and \ref{fig:densTR}}.

\begin{figure}%[h]
\begin{center}
\includegraphics[scale=0.55]{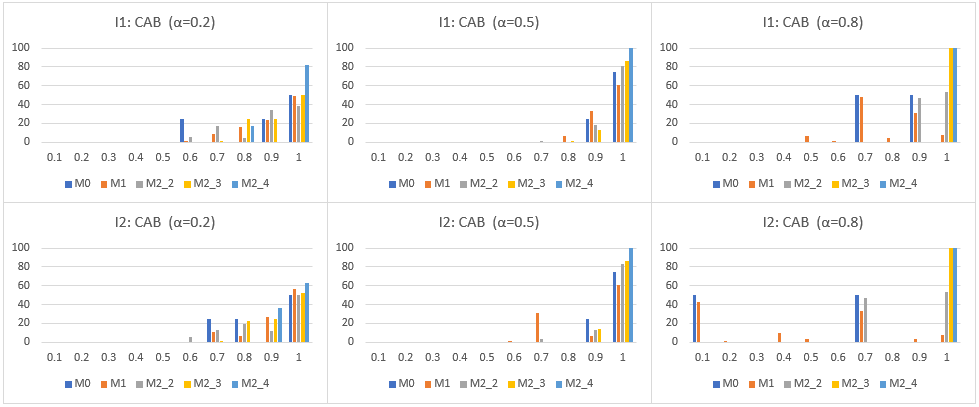}
\caption{Backbone network density for {CAB-based instances} \label{fig:densCAB}}
\end{center}
\end{figure}

\begin{figure}%[h]
\begin{center}
\includegraphics[scale=0.55]{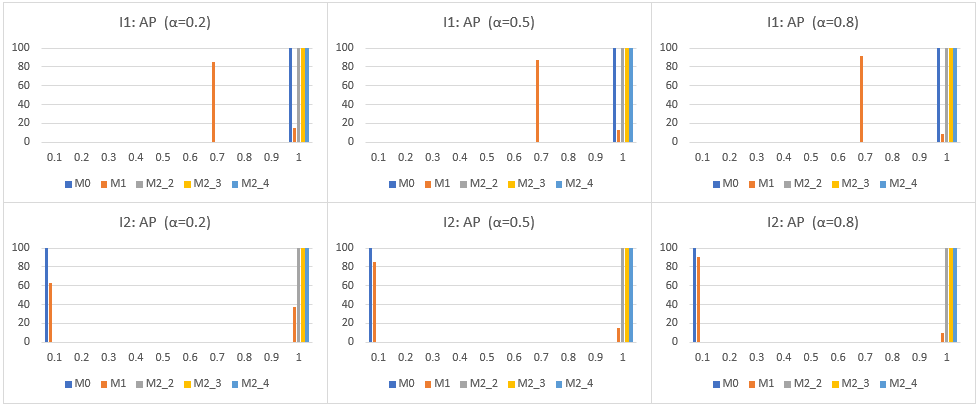}
\caption{Backbone network density for {AP-based instances} \label{fig:densAP}}
\end{center}
\end{figure}

\begin{figure}%[h]
\begin{center}
\includegraphics[scale=0.55]{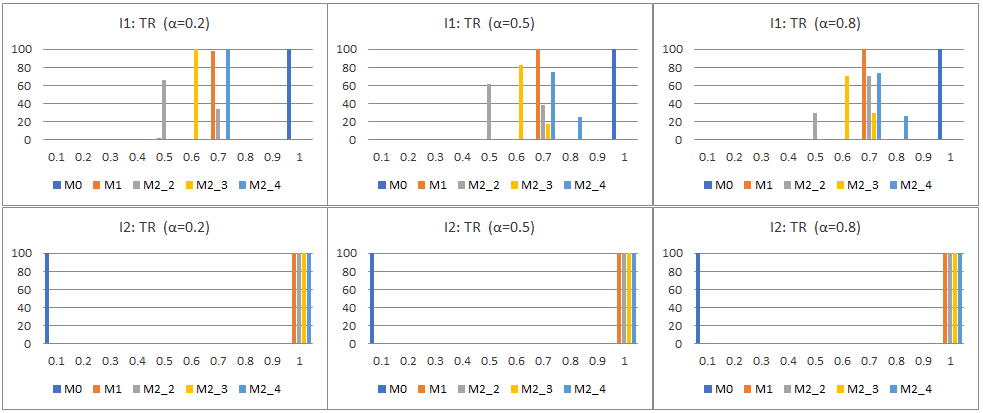}
\caption{Backbone network density for {TR-based instances} \label{fig:densTR}}
\end{center}
\end{figure}

In Table  \ref{tab:nhub}, `` $\#$H'',  ``$\#$Lk'' and ``$\#$Lp'' stand for the average number {of hubs} activated  in the optimal backbone network, the average number of  activated inter-hub edges (including loops), and the average number of activated {loops}, respectively. Two density indices have been studied.
While the first index ($I_1$) indicates the density of the backbone network including loops, the second one ($I_2$) indicates the density of the backbone network when loops are not considered.
 {That is, $I_1$ is} the ratio between the number of activated inter-hub edges and the {total} number of links in a complete graph with $\#H$ nodes if loops were included:
$$I_1=\frac{2 \# Lk}{\#H (\#H+1)},$$
and {$I_2$}  the ratio between the number of {non-loop  inter-hub edges activated} and the number of links in a complete graph with $\#H$ nodes if loops were excluded:
$$I_2=\frac{2 (\# Lk-\#Lp)}{\#H (\#H-1)}.$$
Note that values of the two indices {range in $[0, 1]$, i.e., $0\leq I_1, I_2\leq 1$}.  {Figures \ref{fig:densCAB}, \ref{fig:densAP} and \ref{fig:densTR} show, for each formulation and value of the parameter $\alpha$, the percentage of instances that reach the corresponding index value, for the CAB, the AP and the TR instances, respectively.}

We can observe in Table \ref{tab:nhub} that the number of open hubs depends mainly on the model and also on the dataset. For all AP instances  the number of opens hubs is always {one for M0, two for M1  and M2\_2, three  for  M2\_3,  and four for M2\_4}.   For CAB instances this number ranges in average in [2, 5] for M0, in [2.18, 4.29] for M1, in  [2, 4.76] for M2\_2, in [3, 4.76] for M2\_3, and in [4, 4.74] for M2\_4. For TR instances, the number of opens hubs is always one for M0, two for M1, three for  M2\_2, four  for  M2\_3,  and five for M2\_4.
 On the other hand, we can observe that the number of activated inter-hubs edges {is smaller for M0 than for M1} and that, {for M1}, this number is similar to {that for} M2\_2 but smaller than {that for} M2\_$\lambda$ for $\lambda>2$, {since, as expected, this number increases} with the value of  $\lambda$. Additionally, we can observe that with M0 most of the activated inter-hubs edges are loops, mainly {with} the AP and the TR instances.  This fact can be also observed in Figure  \ref{fig:densAP} for the AP instances where for M0 and M1  the density index $I_1$  is close to 1, whereas the  density index $I_2$ is close to 0. This indicates that most of the activated inter-hubs edges are loops. Figure \ref{fig:densTR} also shows that for the TR instances and  for M0 the optimal backbone network {has a density index $I_1=1$, but the value  of density index $I_2=0$.  On the other hand, we can observe { in  Figure \ref{fig:densCAB}, Figure \ref{fig:densAP}, and Figure \ref{fig:densTR}}, that the density of the backbone network for {M2\_$\lambda$}  increases, as expected, with the value of the parameter $\lambda$.

\section{The price of robustness \label{sec:comput_simul}}

For assessing the robustness and reliability of the hub network models proposed in this paper, we evaluate the so-called price of robustness (see \cite{bertsimas2004price}), defined as the extra cost incurred to design a robust network. In our case robustness translates into protecting the backbone network under inter-hub edge failures,  which, essentially, is attained by incorporating additional inter-hub edges to the backbone network. We thus start out analysis by comparing the overall set-up cost of the activated inter-hub edges for each of the proposed models M1, {M2\_2, M2\_3, and M2\_4}, with that of the \emph{unprotected} network obtained with M0. This information is summarized in Table \ref{t:PoR}, which, for each of the models M1, {M2\_2, M2\_3, and M2\_4}, gives the average percent deviation of the inter-hub set-up costs with respect those of M0. For the sake of simplicity, in this table we only show the results for the TR dataset, although the behavior of the CAB and AP datasets is similar.

% Table generated by Excel2LaTeX from sheet 'Sheet2'
\begin{table}%[htbp]
  \centering
    \begin{tabular}{c|c|rrrr}
    $n$&  $\alpha$ & M1 & M\_2 & M2\_3 & M2\_4 \\\hline
    \multirow{3}[0]{*}{10} & 0.2   & 68.57\% & 136.86\% & 235.11\% & 361.15\% \\
          & 0.5   & 68.57\% & 134.70\% & 235.11\% & 360.23\% \\
          & 0.8   & 68.57\% & 148.34\% & 240.17\% & 358.39\% \\\hline
    \multirow{3}[0]{*}{15} & 0.2   & 68.57\% & 147.33\% & 233.01\% & 358.39\% \\
          & 0.5   & 68.57\% & 134.45\% & 233.01\% & 358.39\% \\
          & 0.8   & 68.57\% & 143.17\% & 234.23\% & 358.39\% \\\hline
    \multirow{3}[0]{*}{20} & 0.2   & 61.52\% & 98.44\% & 191.38\% & 302.97\% \\
          & 0.5   & 61.93\% & 113.60\% & 184.87\% & 279.68\% \\
          & 0.8   & 61.52\% & 126.21\% & 184.87\% & 280.18\% \\\hline
    \multirow{3}[0]{*}{25} & 0.2   & 65.80\% & 139.29\% & 231.72\% & 347.47\% \\
          & 0.5   & 61.52\% & 134.65\% & 223.84\% & 314.85\% \\
          & 0.8   & 61.52\% & 134.52\% & 219.54\% & 314.85\% \\\hline
    \end{tabular}%
  \caption{Average percent deviation of activated hubs and inter-hub set-up costs relative to those of M0. \label{t:PoR}}
%  \caption{Price of Robustness. \label{t:PoR}}
\end{table}%

As expected, constructing networks that are robust under inter-hub edge failures has a significative impact in the design cost of the network. For the $\lambda$-connected models, the reported percent deviations increase with the value of $\lambda$, which can be easily explained as larger backbone networks are  required as $\lambda$ increases. Nevertheless, as shown by the results of the experiment that we report next, in case of failure, this increase in the design costs strengthens the possibility of being able of re-routing all the commodities of  \emph{a posteriori} solutions (after the occurrence of a failure in the inter-hub edges).

\begin{figure}%[H]
\begin{center}
\includegraphics[scale=0.8]{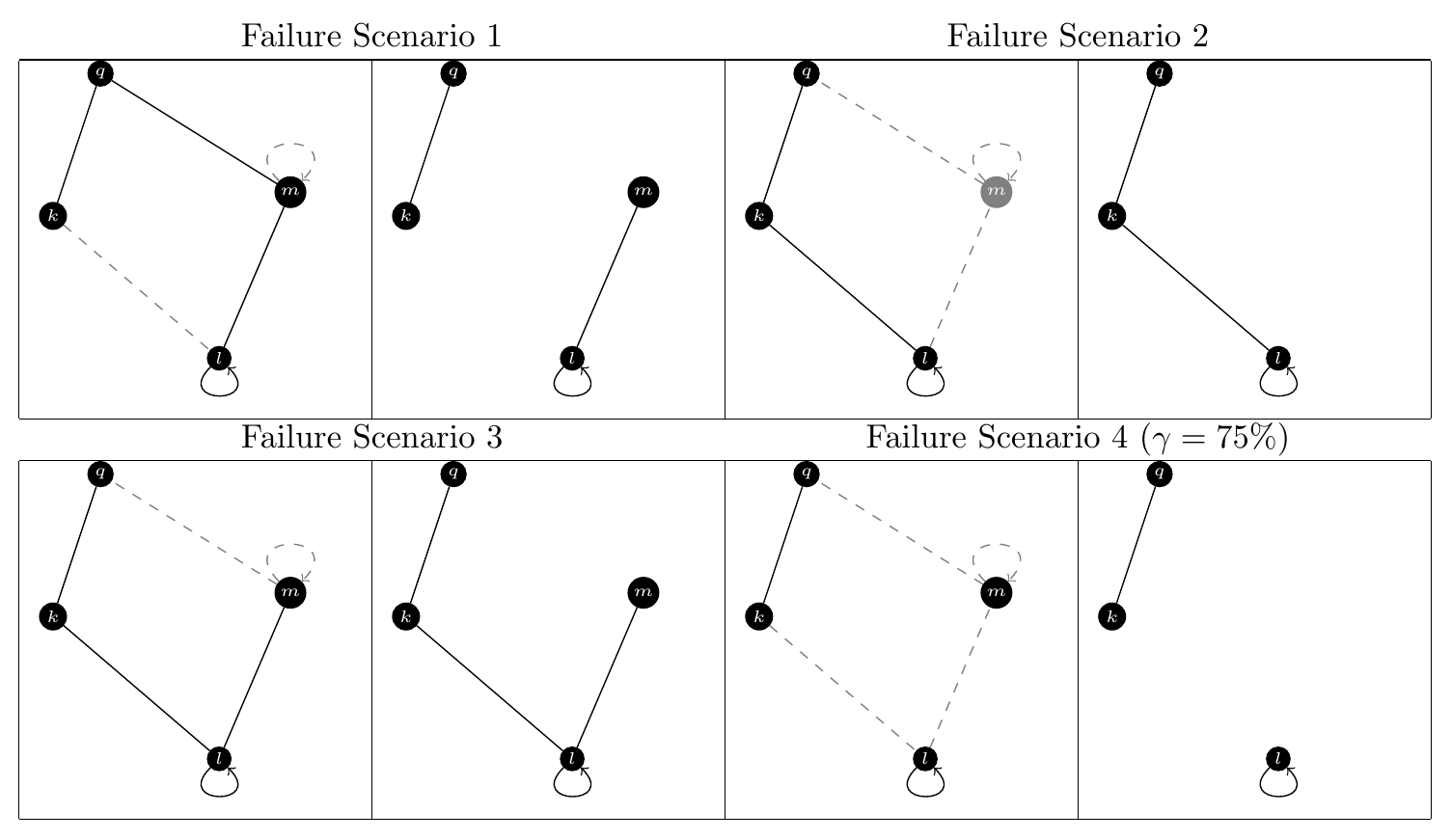}
\caption{Illustration of the four different failure scenarios that we perform.\label{t:simulation}}
\end{center}
\end{figure}

For this experiment, we have used all the $n=20$ instances of the TR dataset, whose optimality is guaranteed for all the models.  For each of the 255 instances generated for the TR dataset with $n=20$, we have simulated the following scenarios for potential failures of the inter-hub edges of the backbone networks produced by the different models:
\begin{itemize}
\item {\bf Failure scenario 1 (FS1):}  Only activated inter-hub edges may fail. Each activated hub edge $\{k,l\}$ in a backbone network fails (and removed from the backbone network) according to a Bernoulli distribution with probability $p_{kl}$.
\item {\bf Failure scenario 2 (FS2):} In this scenario failures are associated with hub nodes. Failure of a hub node implies the failure of all the inter-hub edges incident to the hub. Thus, each activated hub fails with probability $p_{kk}$, and then all its incident inter-hub edges are removed from the backbone network.
\item {\bf Failure scenario 3 (FS3):} Failures are simulated for inter-hub edges of the backbone network similarly to FS1. In case an inter-hub edge fails, the failure probability of the loops at the extreme nodes of the edge is increased in $50\%$. Then, failures in the loop edges are simulated.
\item {\bf Failure scenario 4 (FS4):} First, failures are simulated for inter-hub edges of the backbone network similarly to FS1. The difference is that we now assume that the failure of a considerable number of inter-hub edges incident with any activated hub node, will provoke the failure of the hub node as well, and thus the failure of all its incident inter-hub edges. That is, for each activated hub node, we assume that if at least a given percentage $\gamma\%$ of its incident inter-hub edges have failed, then the whole hub node fails, provoking that its remaining incident inter-hub edges also fail, which are also removed from the network. In our study we fix the value of the parameter $\gamma$ to $75\%$, i.e., {if $75\%$ or more of the inter-hub edges incident to a hub node fails, then, the hub (and the remaining incident inter-hub edges) cannot be used any longer for routing the commodities.}
\end{itemize}

\begin{figure}
     \centering
         \includegraphics[width=0.5\textwidth]{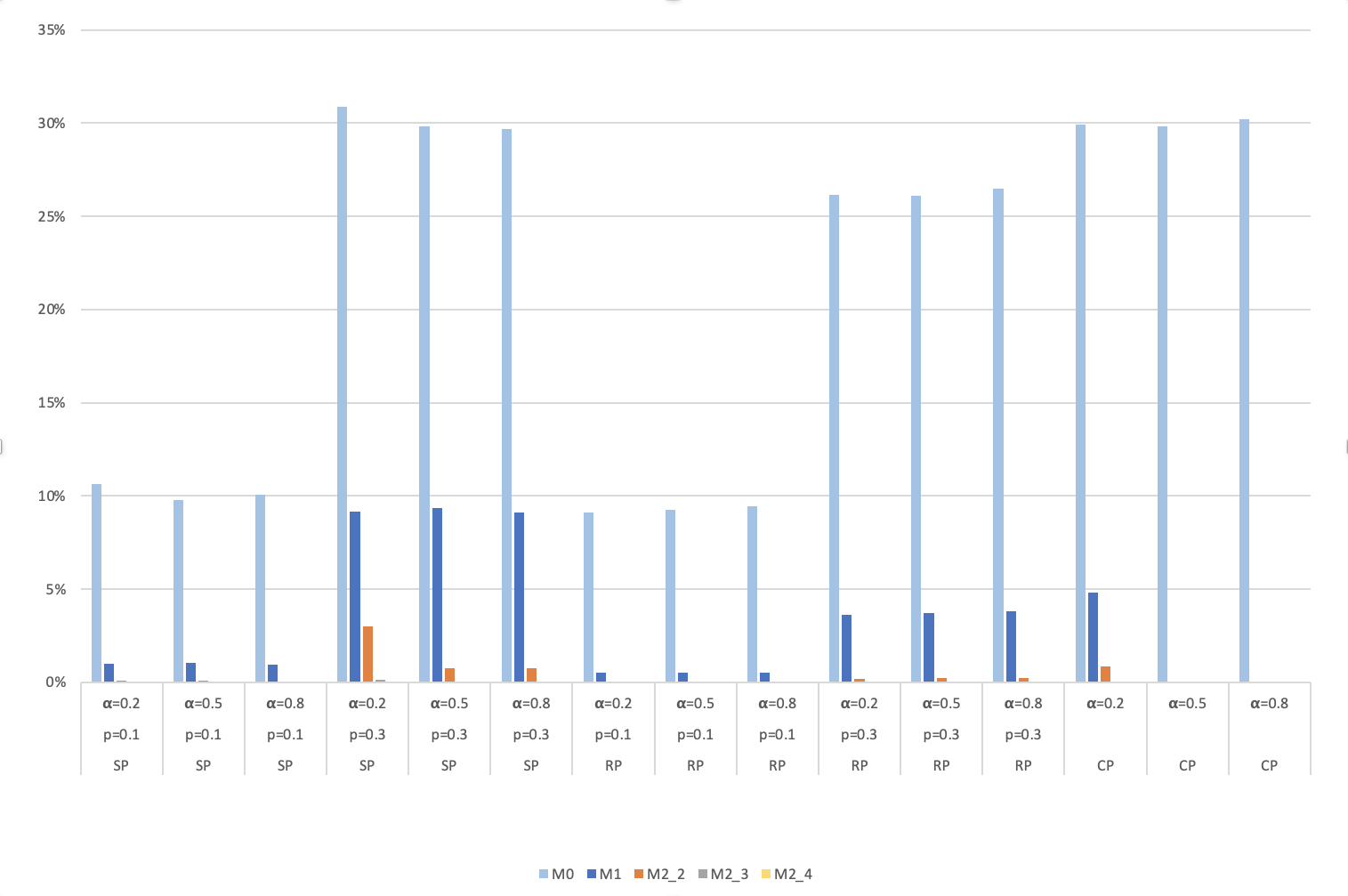}~\includegraphics[width=0.5\textwidth]{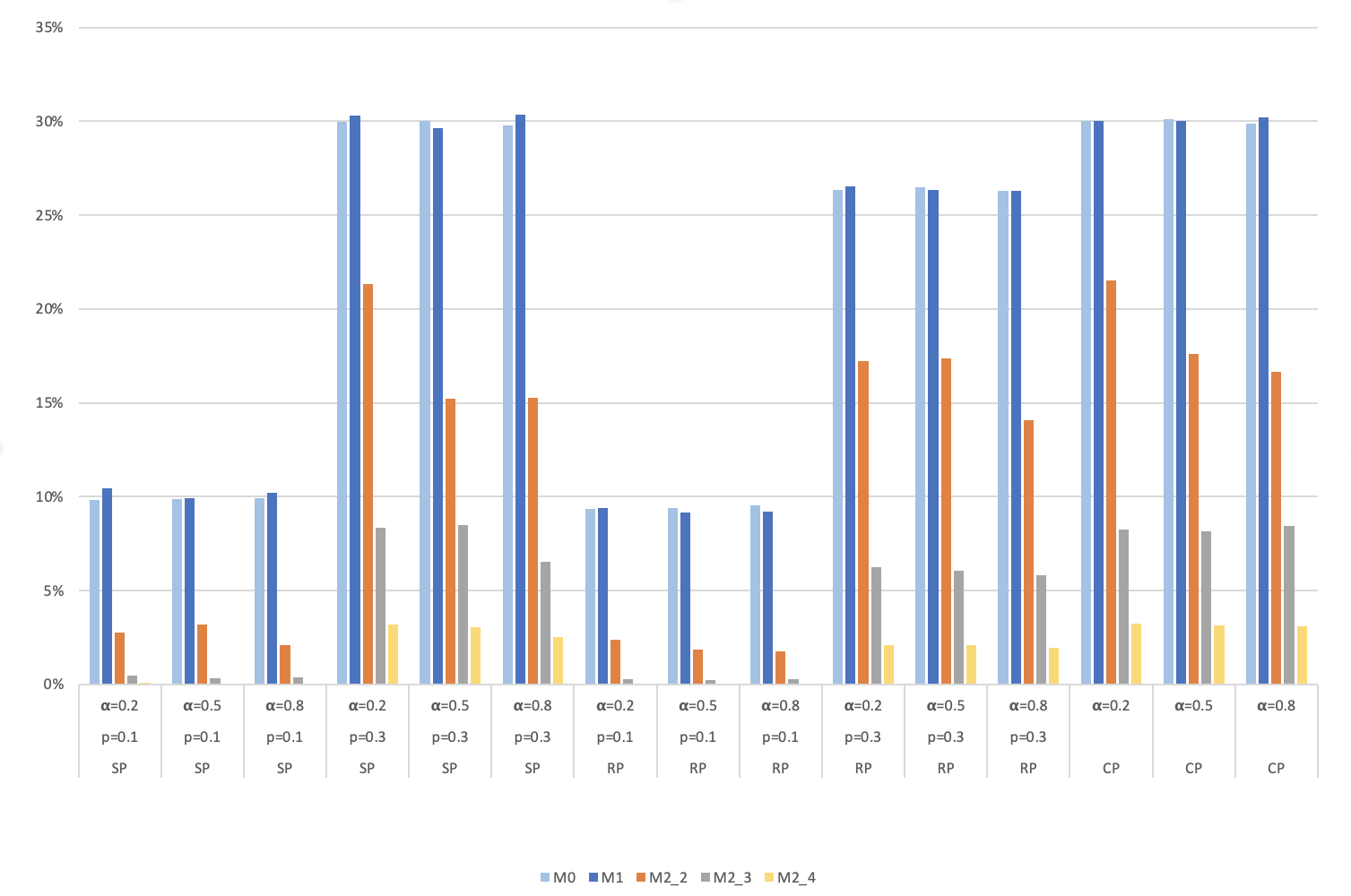}\\
         \includegraphics[width=0.5\textwidth]{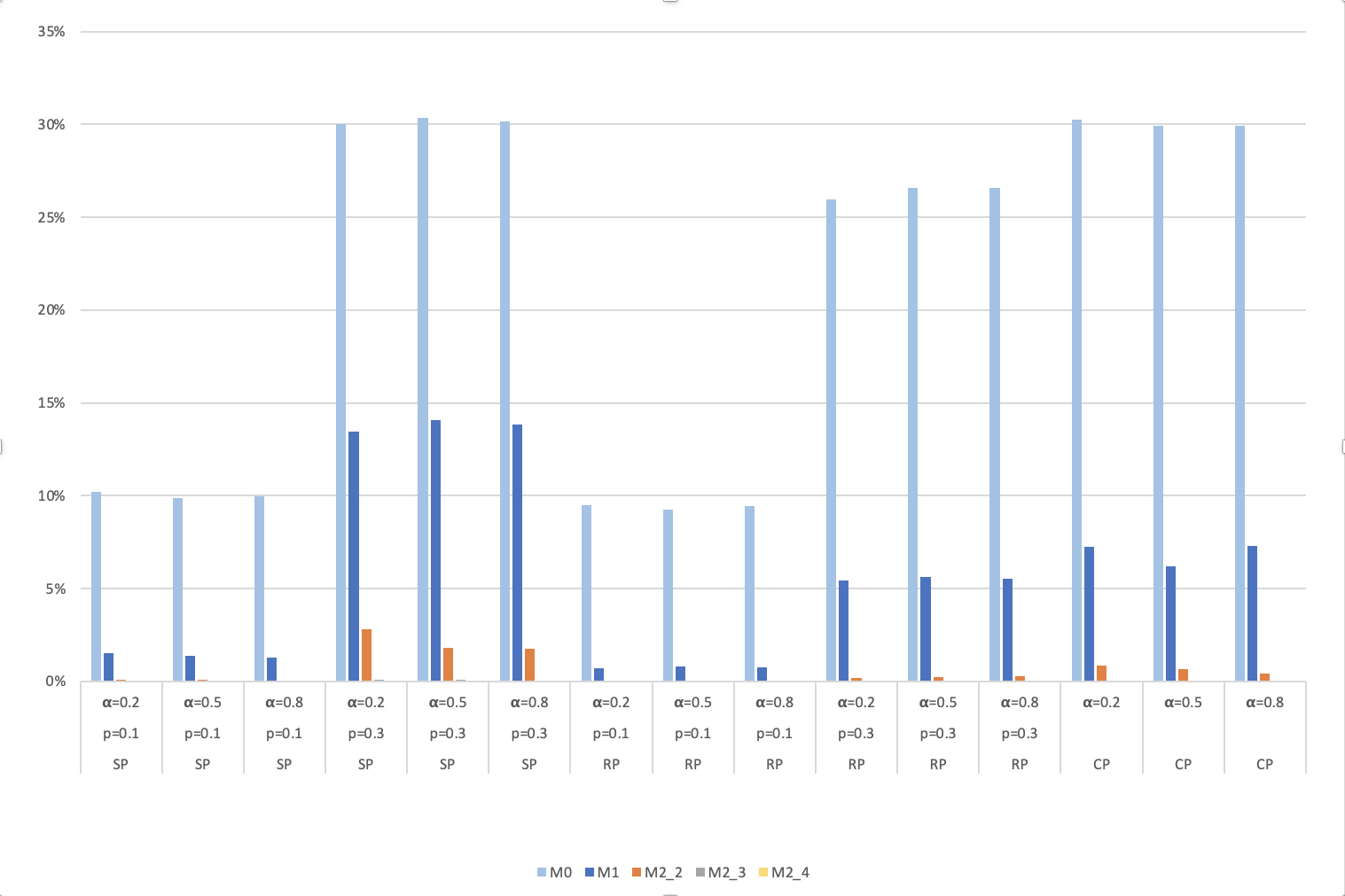}~\includegraphics[width=0.5\textwidth]{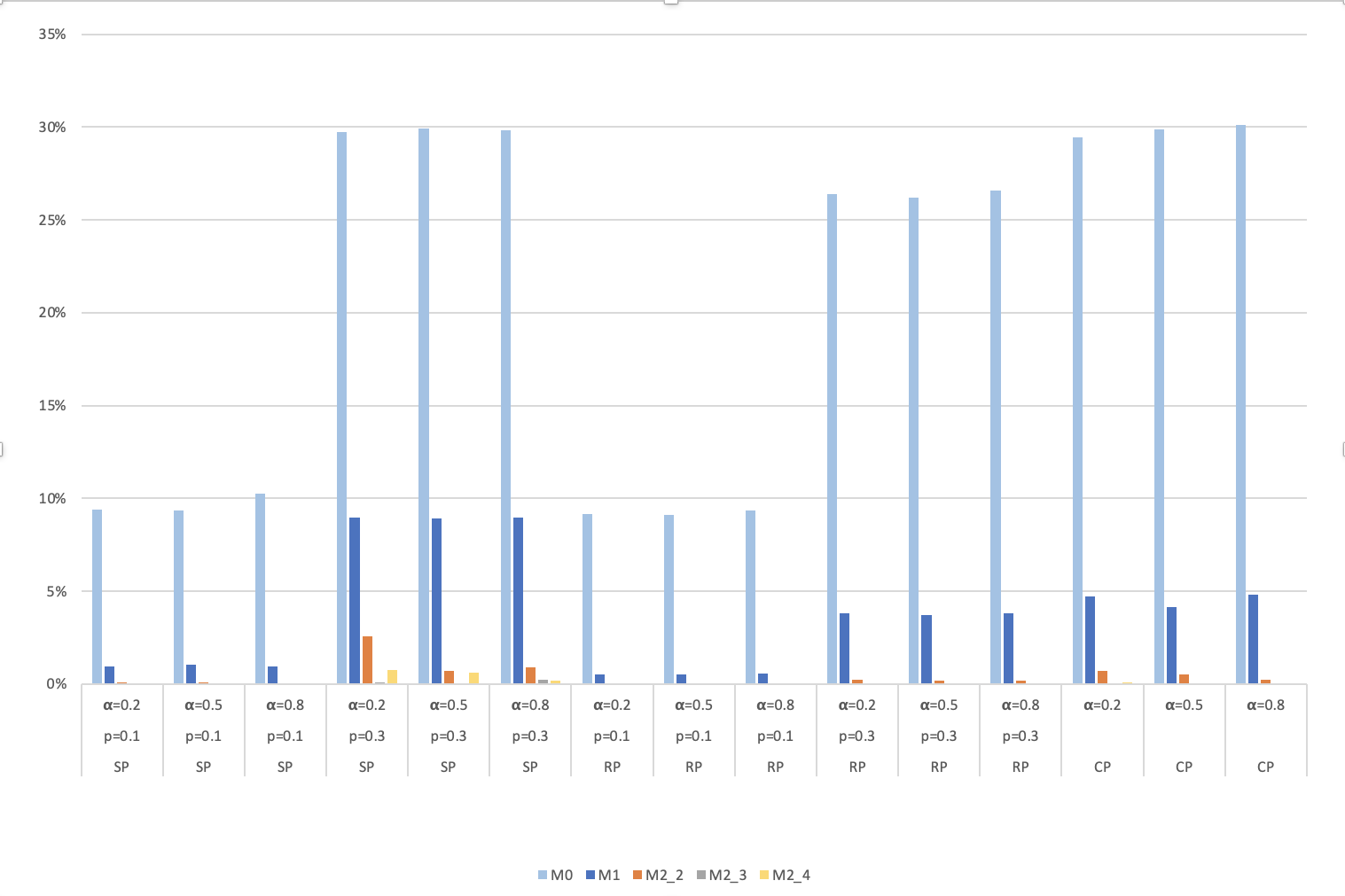}
        \caption{Average percentage of \emph{after-failure} networks for which commodities can no longer be routed (from top left to bottom right: FS1, FS2, FS3, FS4). {Light blue bars represent model $M0$, dark blue model $M1$, orange model $M2\_2$, gray model $M2\_3$ and yellow model $M2\_4$}.\label{fig:FS}}
        \end{figure}

Figure \ref{t:simulation} illustrates with simple backbone networks {($H,E_H$)} the four failure scenarios that we consider. The backbone network has four hub nodes {$H=\{k,l,m,q\}$} and six inter-hub edges, of which  two are loops  {$E_H=\{\{k,l\}$, $\{k,q\}, \{l,m\}, \{q,m\}, \{l,l\},\{m,m\}\}$}. In FS1, hub edges $\{k,l\}$ and  $\{m,m\}$ are chosen to fail, both depicted with dashed lines in the left picture. The network resulting after removing these inter-hub edges  is shown in the right picture. In FS2, hub node $m$ (in a gray circle) is chosen to fail, and then, all the hub edges incident with it (depicted with dashed lines), namely $\{q,m\}, \{m,m\}$ and $\{l,m\}$, removed from the network. In FS3 the interhub edge $\{q,m\}$ is chosen to fail, and increases in $50\%$ the failure probability of the loop $\{m,m\}$ ,which is then randomly chosen to fail. Thus, we chose $\{q,m\}$ and $\{m,m\}$ (depicted with dashed lines), which are then removed from the network. Finally, in FS4, hub edges $\{k,l\}$, $\{l,m\}$ and $\{m,m\}$ are chosen to fail. Then, since the percentage of inter-hub edges incident to $m$ that fail exceeds $\gamma=75\%$, all edges incident to $m$ are removed. For the remaining hub nodes, such a percentage is not exceeded so no further inter-hub edges are removed.
%\begin{figure}[H]
%\begin{center}
%\includegraphics{simulations}
%\caption{Illustration of the four different failure scenarios that we perform.\label{t:simulation}}
%\end{center}
%\end{figure}

We have carried out simulations for each of the above failure scenarios, all of which follow the same general structure, for a given backbone network.  {($i$)} We randomly generate the links that fail according to the corresponding failure scenario and obtain the \emph{after-failure} network by removing from the backbone network the edges that fail. ($ii$) We try to re-route all the commodities through the \emph{after-failure} network. ($iii$) Since it may happen that it is no longer possible to route some of the commodities in the \emph{after-failure} network, we will analyze this circumstance in our study. For each failure scenario, each simulation is repeated $10000$ times over each instance. The average results obtained for all the instances are reported in Figure \ref{fig:FS}. There we draw light blue bars to represent the results for model $M0$, dark blue for$M1$, orange for $M2\_2$, gray for$M2\_3$, and yellow for $M2\_4$.

%\begin{figure}
%     \centering
%         \includegraphics[width=0.5\textwidth]{FS1}~\includegraphics[width=0.5\textwidth]{FS2}\\
%         \includegraphics[width=0.5\textwidth]{FS3}~\includegraphics[width=0.5\textwidth]{FS4}
%        \caption{Average percentage of \emph{after-failure} networks for which commodities can no longer be routed (from top left to bottom right: FS1, FS2, FS3, FS4). {Light blue bars represent model $M0$, dark blue model $M1$, orange model $M2\_2$, gray model $M2\_3$ and yellow model $M2\_4$}.\label{fig:FS}}
%        \end{figure}

As one can observe, the networks obtained with the proposed models (M1 and M2\_$\lambda$) are clearly more robust under inter-hub failures than M0. Specifically, in average, our models allow re-routing all the involved commodities in more than $90\%$ of the failure occurrences while M0 was only able to re-route $78\%$ of them. The robustness of model M2\_4 is even more impressive, being the percentage of simulations in which re-routing is possible $99.5\%$.

On the other hand, analyzing the results of the failure scenario FS2, one can observe that models $M2\_\lambda$ are not only robust under inter-hub edges failures, but also under failures of the hub nodes.  However, $M0$ and $M1$ have a similar behaviour under these scenarios with close to $20\%$ of simulations, in average, in which the commodities could not be routed. {This highlights the performance of
$M2\_2$, $M2\_3$, and $M2\_4$, for which the percentage of simulations where some commodity could not be rerouted decreases to 12\%, 5\%, and 2\%, respectively.}

At each of the simulations, when all commodities can be routed in the \emph{after-failure} network, we compute the overall \emph{a posteriori} routing cost $R^F$. We denote by $\tau_F \in [0,1]$ the proportion of simulations for which this cost can be computed. In case a commodity $r \in R$ is not able to be routed through the \emph{after-failure} backbone network, we assume that its routing cost is proportional to the \emph{cost of the direct connection} $\bar c_{o(r) d(r)}$, i.e. the overall routing cost is $(1+q) \dsum_{r\in R} w_r c_{o(r) d(r)}$. {The parameter $q>0$ represents the \emph{extra percent cost} (over the cost of the direct connection) when re-routing a commodity in case the backbone network cannot be used any longer to route it. Such a cost may represent the outsourcing cost of a direct delivery between the origin and destination of the commodity or the lost of opportunity cost of a unsatisfied user for which the service could not be provided. With this information, we compute the average set-up and routing cost for the network as:
$$
\Phi(q) = \dsum_{h \in H} f_h + \dsum_{e\in E_H} h_e + \tau_F R^F+(1-\tau_F) (1+q) \dsum_{r\in R} w_r c_{o(r) d(r)}.
$$
\begin{figure}%[h]
\centering\includegraphics[scale=0.7]{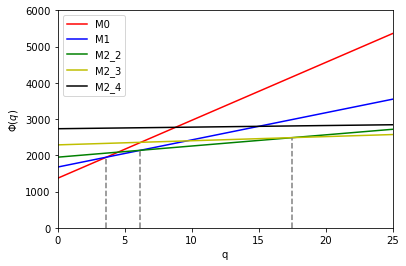}
\centering\caption{Routing costs + Set-up costs on the after-failure network as a function of the parameter $q$.\label{fig:qall}}
\end{figure}

We summarize in Figure \ref{fig:qall} the average behavior of this cost for all the simulations and all the failure scenarios. Each line represents the above cost, as a function of the parameter $q$, for each of the \emph{after-failure} networks produced by the simulations constructed with the five different models (M0, M1, M2\_2, M2\_3, and M2\_4).
One can observe that for small values of $q$ (the re-routing costs are a small factor of the direct costs from origin to destination) $M0$ is more convenient. This is clear, since in case the re-routing costs are not very high, one may undertake these costs, even when these failures occur very often. As $q$ increases, the most convenient models are M1, M2\_2, and M2\_3 (in this order). Model M2\_4 is clearly the most robust one, since the parameter $q$ almost does not affect the cost (in this case the percentage of simulations for which the commodities cannot be routed is tiny), but their set-up costs are very high.

\begin{figure}%[h]
\centering\includegraphics[width=0.5\textwidth]{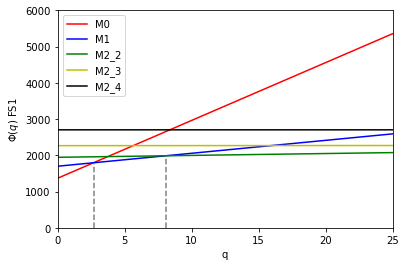}~\includegraphics[width=0.5\textwidth]{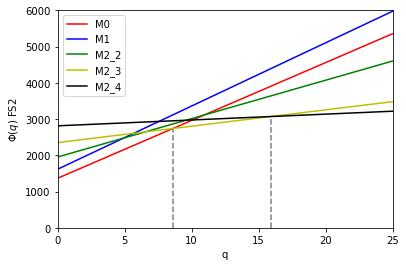}\\
\centering\includegraphics[width=0.5\textwidth]{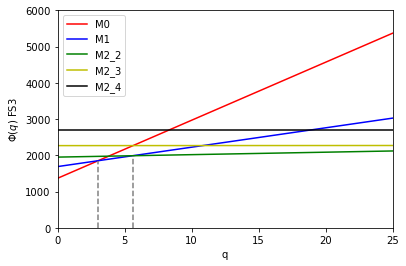}~\includegraphics[width=0.5\textwidth]{figq_3}\\
\centering\caption{Routing costs on the after-failure backbone network as a funciton of the parameter  $q$ for each failure scenario.\label{fig:q}}
\end{figure}

In Figure \ref{fig:q} we show the disaggregated results by failure scenario (FS1, FS2, FS3, and FS4). There, we observe that the behavior of $\Phi(q)$ is different for the failure scenario FS2, where we simulate failures in hub nodes. In FS2, models M0, M2\_3, and M2\_4 outperform M1 and M2\_2, in average, while in the remaining failure scenarios models M0, M1 and M2\_2 are more convenient for reasonable values of $q$.

%\begin{figure}%[h]
%\centering\includegraphics[scale=0.6]{figq_1}~\includegraphics[scale=0.6]{figq_2}\\
%\centering\includegraphics[scale=0.6]{figq_3}~\includegraphics[scale=0.6]{figq_3}\\
%\centering\caption{Routing costs on the after-failure backbone network as a funciton of the parameter  $q$ for each failure scenario.\label{fig:q}}
%\end{figure}

We conclude this section by highlighting that the study that we have carried out allows the decision maker to determine
 the best model to construct the backbone network based on the expected extra cost that should be paid for not providing the service to commodities due to failures in the network.

\section{Conclusions \label{sec:conclu}}

In this paper we propose different models to construct robust hub networks
under inter-hub links failures. The models that we develop ensure that an additional routing
path exists besides its original routing path for all the commodities. In the first model, an explicit backup path using at most an inter-hub edge is constructed to be used in case of failure of the original path from which each of the commodities is routed. The second model assures the existence of backup paths (using an arbitrary number of inter-hub links) in case of failure of the original inter-hub edges by means of imposing $\lambda$-connectivity of the backbone network for a given value of $\lambda\ge 2$. The two models present advantages from the point of view of the robustness of the hub network. One the one hand, the first model guarantees that backup paths for the commodities are of the same nature than the original non-failing network, although the computational difficulty to obtain solutions is high. On the other hand, the second model, although ensuring also the construction of backup paths, exhibits a lower computational load than the first model.

Both models have been computationally tested on an extensive battery of experiments with three hub location benchmarks, namely AP, CAB and TR. Some conclusions are derived from this study. Furthermore, we have analyzed the robustness of the model by simulating different types of failures on the TR network, concluding the applicability of our models.

Future research on the topic includes the study of valid inequalities for both models in order to alleviate the computational complexity of the exact resolution of the model. For larger instances, it would be helpful to design heuristic approaches that assure good quality solution in smaller computing times.

\section*{Acknowledgements}

The authors of this research acknowledge financial support by the Spanish Ministerio de Ciencia y Tecnolog\'ia, Agencia Estatal de Investigaci\'on and Fondos Europeos de Desarrollo Regional (FEDER) via projects PID2020-114594GB-C21 and MTM2019-105824GB-I00. The authors also acknowledge partial support from projects FEDER-US-1256951, Junta de Andaluc\'ia P18-FR-422, P18-FR-2369, B-FQM-322-UGR20 (COXMOS), and NetmeetData: Ayudas Fundaci\'on BBVA a equipos de investigaci\'on cient\'ifica
2019.

\bibliographystyle{apalike}

\bibliography{HLPF_Arxiv}

\end{document}